\documentclass[12pt]{amsart}
\usepackage{amssymb}
\usepackage{amsthm}
\usepackage{amsmath}
\usepackage[dvips]{epsfig}
\usepackage{graphicx}
\usepackage{color}
\usepackage{url}
\usepackage{tikz}
\usepackage{caption}
\usepackage{subcaption}
\usepackage{epstopdf}
\usepackage[arrow, matrix, curve]{xy}

\usetikzlibrary{decorations.markings}

\makeatletter\@addtoreset {equation}{section}\makeatother

\theoremstyle{plain}
\newtheorem{definition}{Definition}[section]

\newtheorem{theorem}[definition]{Theorem}
\newtheorem{corollary}[definition]{Corollary}
\newtheorem{lemma}[definition]{Lemma}

\newtheorem{remark}[definition]{Remark}

\DeclareMathOperator{\sech}{sech}

\newcommand{\norm}[2]{\lVert{#1}\rVert_{#2}}
\newcommand{\abs}[1]{\lvert{#1}\rvert}

{\begin{trivlist} \item[]{\em Proof }}%
{\hspace*{\fill}$\rule{.3\baselineskip}{.35\baselineskip}$\end{trivlist}}

\begin{document}

\title[Drift of spectrally stable shifted states on star graphs]
{Drift of spectrally stable shifted states on star graphs}

\author{Adilbek Kairzhan}
\address{Department of Mathematics, McMaster University, Hamilton, Ontario  L8S 4K1, Canada}
\email{kairzhaa@math.mcmaster.ca}

\author{Dmitry E. Pelinovsky}
\address{Department of Mathematics, McMaster University, Hamilton, Ontario  L8S 4K1, Canada}
\email{dmpeli@math.mcmaster.ca}

\author{Roy H. Goodman}
\address{Department of Mathematical Sciences, New Jersey Institute of Technology, Newark, NJ 07102, USA}
\email{goodman@njit.edu}

\date{\today}

\begin{abstract}
When the coefficients of the cubic terms match the coefficients in the boundary conditions at a vertex
of a star graph and satisfy a certain constraint, the nonlinear Schr\"{o}dinger (NLS) equation on the star graph
can be transformed to the NLS equation on a real line. Such balanced star graphs have appeared in the context
of reflectionless transmission of solitary waves. Steady states on such balanced star graphs can be translated along
the edges with a translational parameter and are referred to as the shifted states.
When the star graph has exactly one incoming edge and several outgoing edges, the steady states are spectrally stable
if their monotonic tails are located on the outgoing edges. These spectrally stable states are degenerate
minimizers of the action functional with the degeneracy due to the translational symmetry.
Nonlinear stability of these spectrally stable states has been an open problem up to now. In this paper,
we prove that these spectrally stable states are nonlinearly unstable because of the irreversible
drift along the incoming edge towards the vertex of the star graph.
When the shifted states reach the vertex as a result of the drift,
they become saddle points of the action functional, in which case the nonlinear instability
leads to their destruction. In addition to rigorous mathematical results,
we use numerical simulations to illustrate the drift instability and destruction of the shifted states
on the balanced star graph.
\end{abstract}

\maketitle


\section{Introduction}

{\em The classical Lyapunov method} establishes stability of minimizers of energy in the time flow
of a dynamical system under the condition that the second variation of energy is strictly positive definite.
In the context of Hamiltonian PDEs such as the nonlinear Schr\"{o}dinger (NLS) equation,
standing waves are often saddle points of energy but additional conserved quantities such as
mass and momentum exist due to symmetries such as phase rotation and space translation.
It is now a classical result \cite{GSS87,W86} that the standing waves are stable if
they are constrained minimizers of energy when other
conserved quantities are fixed and if the second variation of energy is strictly positive definite
under the constraints eliminating symmetries.

Stability of standing waves in the presence of symmetries is understood in the sense of {\em orbital stability},
where the orbit is defined by a set of parameters along the symmetry group.
Fixed values of other conserved quantities are realized by Lagrange multipliers which
define another set of parameters. Both sets of parameters (along the
symmetry group and Lagrange multipliers) satisfy the modulation equations in the time flow of the Hamiltonian PDE \cite{W87}.
The actual values of parameters along the symmetry group are irrelevant in the definition
of an orbit for the standing waves, whereas the actual values of Lagrange multipliers and
the remainder terms are controlled if the second variation of energy is strictly positive
under the symmetry constraints.

The present work is devoted to stability of standing waves in the NLS equation defined on a metric graph,
a subject that has seen many recent developments \cite{Noja}. Existence and variational characterization
of standing waves was developed for star graphs \cite{AdamiJPA,AdamiAH,AdamiJDE1,AdamiJDE2} and
for general metric graphs \cite{AdamiCV,AdamiJFA,AST,CDS,D18}. Bifurcations and stability of standing waves
were further explored for tadpole graphs \cite{NPS}, dumbbell graphs \cite{G19,MP16}, double-bridge graphs
\cite{NRS}, and periodic ring graphs \cite{D19,GilgPS,P18,PS17}. A variational characterization of standing waves was developed
for graphs with compact nonlinear core \cite{T1,T2,T3}.

In the context of the simplest {\em star graphs}, it was realized in \cite{AdamiJFA} that the infimum of energy
under the fixed mass is approached by a sequence of solitary waves escaping to infinity along one edge of the star graph.
Consequently, standing waves cannot be energy minimizers, as they represent saddle points of energy
under the fixed mass \cite{AdamiJPA}. It was shown in \cite{KP1} that the second variation of energy at these
standing waves is nevertheless positive but degenerate with a zero eigenvalue, hence the standing waves
are saddle points beyond the second variation of energy. Orbital instability of these standing waves
under the time flow of the NLS equation is developed as a result of the saddle point geometry \cite{KP1}.

{\em Balanced star graphs} were introduced in \cite{M1,M2} from the condition that
the NLS equation on the metric graph reduces to the NLS equation on a real line
if the initial conditions satisfy certain symmetry.
Consequently, solitary waves may propagate across the vertex without any reflection.
Standing waves of the NLS equation can be translated along edges of the balanced star graph with a translational parameter
and are referred to as {\em the shifted states}. When the star graph has exactly one incoming edge and
several outgoing edges, the shifted states were shown to be spectrally stable if
their monotonic tails are located on the outgoing edges \cite{KP2}. Moreover,
these shifted states are constrained minimizers of the energy under  fixed mass
and the only degeneracies of the second variation of energy are due to
phase rotation and the spatial translation along the balanced star graph.

Standing waves are orbitally stable in the NLS equation on a real line
since the two degeneracies are related to two symmetries of the NLS equation
which also conserves mass and momentum. In contrast to this well-known result, we show
in this paper that {\em the shifted states are orbitally unstable in the NLS equation
on the balanced star graph}.

The instability is related to the following observation. If the initial
perturbation to the standing wave is symmetric with respect to the exchange of
components on the outgoing edges, then the reduction to the NLS equation on a line holds,
and the solution has translational symmetry. Perturbations that lack this exchange symmetry
also break the translation symmetry and the solution fails to conserve momentum.
Moreover, the value of the momentum functional increases monotonically in the time flow of the NLS
equation and this monotone increase results in the irreversible drift
of the shifted state along the incoming edge towards the outgoing edges of the balanced star graph.
When the center of mass for the shifted state reaches the vertex, the shifted state becomes a saddle point of
energy under the fixed mass. At this point in time, orbital instability of the shifted state develops
as a result of the saddle point geometry similar to the instability studied in \cite{KP1}.

The main novelty of this paper is to show that {\em degeneracy of the positive second variation
of energy may lead to orbital instability of constrained minimizers if this degeneracy is not
related to the symmetry of the Hamiltonian PDE}. The orbital instability appears due to
irreversible drift of shifted states from a spectrally stable state towards the spectrally unstable states.
We prove rigorously the conjecture posed in the previous work \cite{KP2}
and confirm numerically the instability of the shifted states on the balanced star graphs.

The paper is structured as follows. Section \ref{sec-main} presents the background material
and the main results of this work. Section \ref{sec-linear} collects together the linear estimates.
Section \ref{sec-theorem-1} gives the proof of the irreversible drift along the spectrally stable shifted states.
Section \ref{sec-theorem-2} gives the proof of nonlinear instability of the limiting shifted state (called {\em the half-soliton state})
due to the saddle point geometry. Section \ref{sec-numerics} illustrates the analytical results with numerical
simulations. Section \ref{sec-conclusion} concludes the paper with a summary.

\section{Main results}
\label{sec-main}

We consider a {\em star graph} $\Gamma$ constructed by attaching $N$ half-lines at a common vertex.
In the construction of the graph $\Gamma$, one edge represents an incoming bond
and the remaining $N-1$ edges represent outgoing bonds. We place the vertex at the origin
and parameterize the incoming edge by $\mathbb{R}^-$ and the $N-1$ outgoing edges by $\mathbb{R}^+$.
An illustration of the star graph $\Gamma$ with one incoming and three outgoing edges
is shown on Fig. \ref{fig-1}.

\begin{figure}[htbp] 
   \centering
   \includegraphics[width=4in, height = 2in]{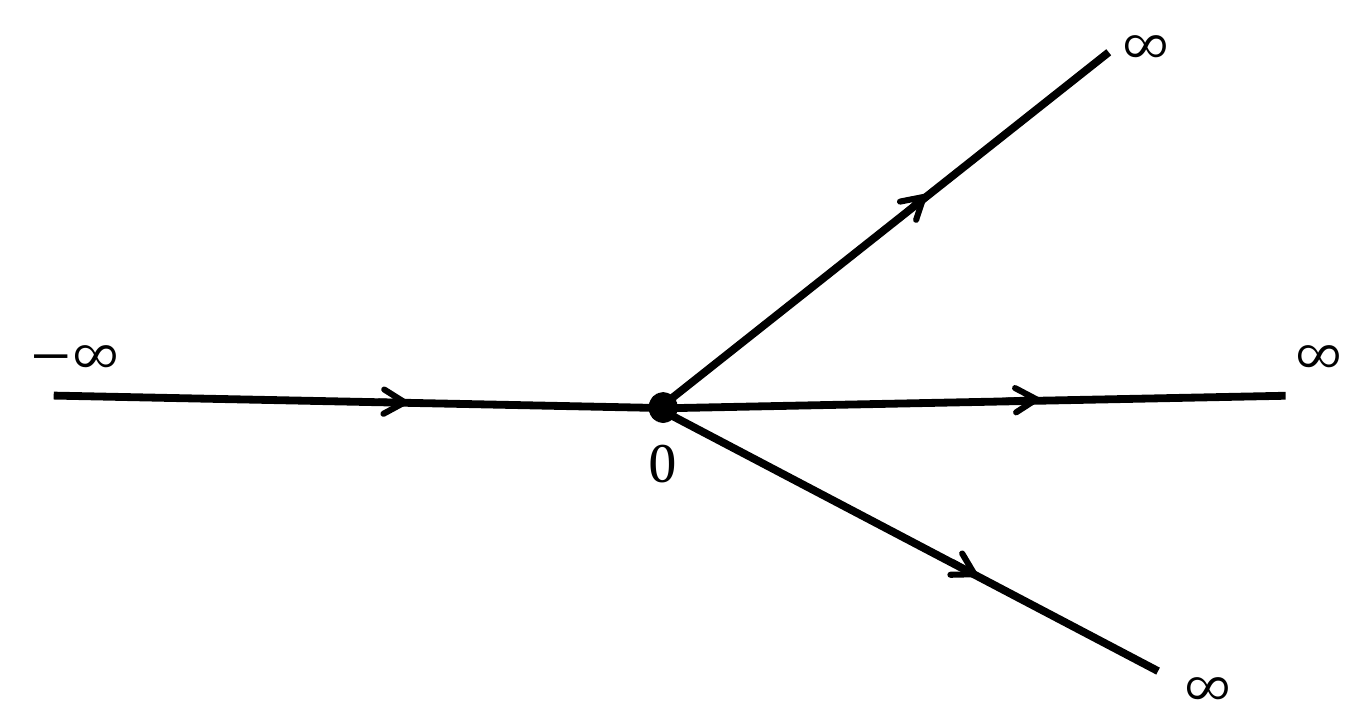}
   \caption{A star graph with $N=4$ edges.}
   \label{fig-1}
\end{figure}

The Hilbert space
$$
L^2(\Gamma) = L^2(\mathbb{R}^-) \oplus \underbrace{L^2(\mathbb{R}^+) \oplus \cdots \oplus L^2(\mathbb{R}^+)}_\text{\rm (N-1) elements}
$$
is defined componentwise on edges of the star graph $\Gamma$.
Sobolev spaces $H_\Gamma^1$ and $H_\Gamma^2$ are also defined componentwise subject
to the generalized Kirchhoff boundary conditions:
\begin{equation}
\label{H1}
H_{\Gamma}^1 := \{ \Psi \in H^1(\Gamma): \quad \alpha_1 \psi_1(0) = \alpha_2 \psi_2(0) = \dots = \alpha_N \psi_N(0)\}
\end{equation}
and
\begin{equation}
\label{H2}
 H_{\Gamma}^2 := \left\{ \Psi \in H^2(\Gamma) \cap H^1_{\Gamma} : \quad
\alpha_1^{-1} \psi_1'(0) = \sum_{j=2}^N \alpha_j^{-1} \psi_j'(0) \right\},
\end{equation}
where derivatives are defined as $\lim_{x\to 0^-}$ for the incoming edge and
$\lim_{x \to 0^+}$ for the $(N-1)$ outgoing edges. The dual space to $H^1_{\Gamma}$ is
$H^{-1}_{\Gamma} \equiv H^{-1}(\Gamma)$, which is also defined componentwise.

We consider a {\em balanced star graph} defined by the following constraint on
the positive coefficients $(\alpha_1, \alpha_2, \cdots, \alpha_N)$ in the boundary conditions (\ref{H1}) and (\ref{H2}):
\begin{equation}
\label{constraint}
\frac{1}{\alpha_1^2} = \sum_{j=2}^N \frac{1}{\alpha_j^2}.
\end{equation}
This constraint was introduced in \cite{M1,M2} as a condition of the
reflectionless transmission of a solitary wave across the vertex
of the star graph $\Gamma$.

The time flow on $\Gamma$ is given by the following nonlinear Schr\"{o}dinger (NLS) equation:
\begin{equation} \label{eq1}
 i \frac{\partial \Psi}{\partial t} = - \Delta \Psi - 2 \alpha^2 |\Psi|^{2} \Psi,
\end{equation}
where $\Psi = \Psi(t,x)$, $\Delta \Psi = (\psi_1'',\psi_2'',\dots,\psi_N'')$ is the Laplacian operator
defined componentwise with primes denoting derivatives in $x$,
$\alpha = (\alpha_1, \alpha_2, \dots, \alpha_N) \in \mathbb{R}^N$
represents the same coefficients as in (\ref{H1}) and (\ref{H2}),
and the nonlinear term $\alpha^2 |\Psi|^{2} \Psi$ is interpreted as a symbol for
$(\alpha_1^2|\psi_1|^{2}\psi_1, \alpha_2^2 |\psi_2|^{2}\psi_2, \dots, \alpha_N^2 |\psi_N|^{2} \psi_N)$.

Local and global well-posedness of the Cauchy problem for the NLS equation (\ref{eq1})
is well-known both for weak solutions in $H^1_{\Gamma}$ and
for strong solutions in $H^2_{\Gamma}$ (see Proposition 2.2 in \cite{AdamiJDE1} and Lemmas 2.2 and 2.3 in \cite{KP2}).
For any of these solutions,
we can define the energy and mass functionals as
\begin{equation}
\label{energy}
E(\Psi) := \| \Psi' \|_{L^2(\Gamma)}^2 - \| \alpha^{\frac{1}{2}} \Psi \|_{L^{4}(\Gamma)}^{4},
\quad Q(\Psi) := \| \Psi \|_{L^2(\Gamma)}^2,
\end{equation}
respectively. These functionals are constants under the time flow of the NLS equation (\ref{eq1}) thanks to the fact that
$\Delta : H^2_{\Gamma} \subset L^2(\Gamma) \to L^2(\Gamma)$ is extended as a
self-adjoint operator in the Hilbert space $L^2(\Gamma)$ (see Lemma 2.1 in \cite{KP2}).

The NLS equation (\ref{eq1}) admit standing wave solutions of the form
$$
\Psi(t,x) = e^{i\omega t} \Phi_{\omega}(x),
$$
where the real-valued pair $(\omega,\Phi_{\omega})$ satisfies the stationary NLS equation,
\begin{equation}\label{eq2}
-\Delta \Phi_{\omega} - 2 \alpha^2 |\Phi_{\omega}|^{2} \Phi_{\omega} = - \omega \Phi_{\omega}, \quad \Phi_{\omega} \in H^2_{\Gamma}.
\end{equation}
The stationary NLS equation is the Euler--Lagrange equation for the action functional
\begin{equation}
\label{lyapunov-funct}
\Lambda_{\omega}(\Psi) := E(\Psi) + \omega Q(\Psi).
\end{equation}
It is also well-known \cite{AdamiJDE1}
that the set of critical points of $\Lambda_{\omega}$ in $H^1_{\Gamma}$ is equivalent
to the set of solutions of the stationary NLS equation in $H^2_{\Gamma}$.

For $\omega > 0$, we can set $\omega = 1$ by employing the following scaling transformation:
\begin{equation}
\label{scal-transf}
\Phi_{\omega}(x) = \omega^{\frac{1}{2}} \Phi(z), \quad z = \omega^{\frac{1}{2}} x.
\end{equation}
The following lemma states the existence of a family of shifted states
in the stationary NLS equation (\ref{eq2}) with the boundary conditions in (\ref{H1}) and
(\ref{H2}), where the coefficients $(\alpha_1, \alpha_2, \dots, \alpha_N)$ satisfy
the constraint (\ref{constraint}).

\begin{lemma}
\label{solutions-II}
For every $(\alpha_1,\alpha_2,\dots,\alpha_N)$ satisfying the constraint (\ref{constraint}), there exists
a unique one-parameter family of solutions $\{ \Phi(x;a) \}_{a \in \mathbb{R}}$ to the stationary NLS equation (\ref{eq2}) with $\omega = 1$,
where each component of $\Phi(x;a)$ is given by
\begin{equation}
\label{soliton-shifted}
\phi_j(x;a) = \alpha_j^{-1} \phi(x+a), \quad 1 \leq j \leq n,
\end{equation}
with $\phi(x) = \sech(x)$.
\end{lemma}

\begin{proof}
The stationary NLS equation (\ref{eq2}) with $\omega = 1$ admits
a general solution $\Phi = (\phi_1, \dots, \phi_N) \in H^2(\Gamma)$ of the form:
$$
\phi_j(x) = \alpha_j^{-1}\phi(x+a_j), \quad j=1,\dots,N,
$$
with $\phi(x) = \sech(x)$. Parameters $(a_1, \dots, a_N) \in \mathbb{R}^N$ are to
be defined by the boundary conditions in (\ref{H1}) and (\ref{H2}).
The continuity condition in (\ref{H1}) implies that
$|a_1| = \dots = |a_N|$, and so, for every
$j=1,\dots, N$, there exists $\sigma_j \in \{-1,1\}$ such that $a_j = \sigma_j a_1$.
Without loss of generality, we choose $\sigma_1 = 1$.
The Kirchhoff condition in (\ref{H2}) implies that, under the constraint (\ref{constraint}),
\begin{equation}
\label{sigma_eqn}
\phi'(a_1) \sum_{j=2}^N \frac{\sigma_j-1}{\alpha_j^2} = 0.
\end{equation}
The equation (\ref{sigma_eqn}) holds if either $\phi'(a_1) = 0$ or
$\sum_{j=2}^N \frac{\sigma_j-1}{\alpha_j^2} = 0$. The first case has a unique solution $a_1=0$.
The second case holds for every $a_1 \in \mathbb{R} \backslash \{0\}$
if and only if for every $j = 2, \dots, N$ we get $\sigma_j=1$, since
$\frac{\sigma_j-1}{\alpha_j^2}$ is either negative or zero. Combining both cases,
we have $a_1=a_2=\dots = a_N=a$, where $a\in \mathbb{R}$ is arbitrary,
as is given by (\ref{soliton-shifted}).
\end{proof}

\begin{remark}
Compared to the parametrization of edges in $\Gamma$ used in our previous work \cite{KP2},
$\phi_1(x)$ is defined here for $x \in \mathbb{R}^-$ while all other $\phi_j(x)$ are defined for $x \in \mathbb{R}^+$.
We also replace parameter $a \in \mathbb{R}$ in \cite{KP2} by $-a \in \mathbb{R}$ for convenience.
\end{remark}

The shifted state (\ref{soliton-shifted}) in Lemma \ref{solutions-II} satisfies the following symmetry.

\begin{definition}
\label{alpha-symmetric}
For every fixed $\alpha = (\alpha_1, \alpha_2, \cdots, \alpha_N)$
we say that the function $\Psi = (\psi_1, \psi_2, \cdots, \psi_N) \in H^1_{\Gamma}$ is
$\alpha$-symmetric if it satisfies for all $x\in \mathbb{R}^+$:
\begin{equation}
\label{symmetry}
\alpha_2 \psi_2(x) = \cdots = \alpha_N \psi_N(x).
\end{equation}
\end{definition}

The symmetry (\ref{symmetry}) in Definition \ref{alpha-symmetric} provides the following reduction
of the NLS equation (\ref{eq1}) on the balanced star graph $\Gamma$ under the constraint (\ref{constraint}).

\begin{lemma}
Assume that $\Psi \in C(\mathbb{R},H^2_{\Gamma}) \cap C^1(\mathbb{R},L^2(\Gamma))$ is a strong
solution to the NLS equation (\ref{eq1}) under the constraint (\ref{constraint}) satisfying the
symmetry reduction (\ref{symmetry}) for every $t \in \mathbb{R}$. The wave function
\begin{equation}
\label{varphi}
\varphi(t,x) = \left\{ \begin{array}{ll} \alpha_1 \psi_1(t,x), \quad & x \leq 0, \\
\alpha_2 \psi_2(t,x), \quad & x \geq 0, \end{array} \right.
\end{equation}
is a strong solution $\varphi \in C(\mathbb{R},H^2(\mathbb{R})) \cap C^1(\mathbb{R},L^2(\mathbb{R}))$
to the following NLS equation on the real line $\mathbb{R}$:
\begin{equation} \label{nls}
 i \frac{\partial \varphi}{\partial t} = - \frac{\partial^2 \varphi}{\partial x^2} - 2 |\varphi|^{2} \varphi.
\end{equation}
\end{lemma}

\begin{proof}
The NLS equation (\ref{nls}) is defined piecewise for $x < 0$ and $x > 0$
from the NLS equation (\ref{eq1}), the symmetry (\ref{symmetry}), and the representation (\ref{varphi}).
Thanks to the boundary conditions in (\ref{H1}) and (\ref{H2}), the function $\varphi(t,x)$ is continuously differentiable
across the vertex point $x = 0$. Hence if $\Psi \in C(\mathbb{R},H^2_{\Gamma}) \cap C^1(\mathbb{R},L^2(\Gamma))$
is a strong solution to (\ref{eq1}),
then $\varphi \in C(\mathbb{R},H^2(\mathbb{R})) \cap C^1(\mathbb{R},L^2(\mathbb{R}))$ is a strong solution to (\ref{nls}).
\end{proof}

The NLS equation (\ref{nls}) on the real line enjoys the translational symmetry in $x$.
The free parameter $a$ in the family of shifted state (\ref{soliton-shifted}) in Lemma \ref{solutions-II}
is related to the translational symmetry of the NLS equation (\ref{nls}) in $x$.
However, the translational symmetry is broken for the NLS equation (\ref{eq1}) on the star graph $\Gamma$
due to the vertex at $x = 0$. As a result, the momentum functional given by
\begin{equation}
\label{momentum}
P(\Psi) := {\rm Im} \langle \Psi',
\Psi \rangle_{L^2(\Gamma)} =
\int_{\mathbb{R}^-} {\rm Im} \left( \psi'_1 \overline{\psi}_1 \right) dx +
\sum_{j=2}^N \int_{\mathbb{R}^+} {\rm Im} \left( \psi'_j \overline{\psi}_j \right) dx
\end{equation}
is no longer constant under the time flow of (\ref{eq1}). It was
shown in \cite{KP2} (see Lemma 6.1) that for every weak solution $\Psi \in C(\mathbb{R},H^1_{\Gamma}) \cap C^1(\mathbb{R},H^{-1}(\Gamma))$
to the NLS equation (\ref{eq1}) the map $t \mapsto P(\Psi)$ is monotonically increasing,
thanks to the following inequality:
\begin{equation}
\label{dmom_dt}
\frac{d}{dt} P(\Psi)  =  \frac{1}{2} \sum_{j=2}^N \sum_{\substack{i=2 \\ i \neq j}}^N \frac{\alpha_1^2}{\alpha_j^2 \alpha_i^2}
\left| \alpha_j \psi_j'(0) - \alpha_i \psi_i'(0) \right|^2 \geq 0.
\end{equation}
If the weak solution $\Psi$ satisfies the symmetry (\ref{symmetry}) in Definition \ref{alpha-symmetric},
then $P(\Psi)$ is conserved in $t$.

It was proved in \cite{KP2} (see Theorem 4.1 and Corollary 4.3) that the one-parameter family of
shifted states in Lemma \ref{solutions-II} is spectrally unstable for $a < 0$ and spectrally stable
for $a > 0$. The degenerate state at $a = 0$ called {\em the half-soliton state} is expected to be nonlinearly unstable
as was shown in \cite{KP1} for uniform star graphs with $\alpha = 1$. In regards to the shifted states with $a > 0$,
it was conjectured in \cite{KP2} (see Conjectures 7.1 and 7.2) that the shifted state in Lemma \ref{solutions-II}
with $a > 0$ is nonlinearly unstable under the time flow because the shifted states drifts
along the only incoming edge towards the half-soliton state with $a = 0$, where
it is affected by spectral instability of the shifted states with $a < 0$.

The present work is devoted to the proof of the aforementioned conjectures. Our first main result shows that
the monotone increase of the map $t \mapsto P(\Psi)$ as in (\ref{dmom_dt})
leads to a drift along the family of shifted states (\ref{soliton-shifted})
in which the parameter $a$ decreases monotonically in $t$ towards $a = 0$.
This drift induces nonlinear instability of the spectrally stable shifted states
in Lemma \ref{solutions-II} with $a > 0$. The following theorem formulates the result.

\begin{theorem}
\label{main_theorem_1}
Fix $a_0>0$. For every $\nu \in (0,a_0)$ there exists $\epsilon_0 > 0$ (sufficiently small) such that
for every $\epsilon \in (0,\epsilon_0)$, there exists $\delta>0$ and $T > 0$ such that
for every initial datum $\Psi_0 \in H^1_\Gamma$ with $P(\Psi_0) > 0$ and
\begin{equation}
\label{theorem_bound_delta}
\inf_{\theta \in \mathbb{R}} \|\Psi_0 - e^{i\theta} \Phi(\cdot;a_0)\|_{H^1(\Gamma)} \leq \delta
\end{equation}
the unique solution $\Psi \in C([0,T], H^1_\Gamma) \cap C^1([0,T], H^{-1}_\Gamma)$
to the NLS equation (\ref{eq1}) with the initial datum $\Psi(0,\cdot) = \Psi_0$ satisfies the bound
\begin{equation}
\label{theorem_bound_epsilon}
\inf_{\theta \in \mathbb{R}} \|\Psi(t,\cdot) - e^{i\theta} \Phi(\cdot;a(t))\|_{H^1(\Gamma)} \leq \epsilon,
\quad t \in [0,T],
\end{equation}
where $a \in C^1([0,T])$ is a strictly decreasing function such that $\lim_{t \to T} a(t) = \nu$.
\end{theorem}

By Theorem \ref{main_theorem_1}, the shifted state (\ref{soliton-shifted})
with $a > 0$ drifts towards the half-soliton state with $a = 0$. The half-soliton state is more degenerate than
the shifted state with $a > 0$ because the zero eigenvalue of the Jacobian operator associated with
the stationary NLS equation (\ref{eq2}) is simple for $a > 0$ and has multiplicity $N-1$ for $a = 0$.
Moreover, while the shifted state $\Phi(\cdot;a)$ with $a > 0$ is a degenerate minimizer of the action functional
$\Lambda_{\omega = 1}(\Psi) = E(\Psi) + Q(\Psi)$, the half-soliton state $\Phi \equiv  \Phi(\cdot;a=0)$ is a degenerate
saddle point of the same action functional \cite{KP1}. The following theorem shows the nonlinear instability of
the half-soliton state related to the saddle point geometry of the critical point.

\begin{theorem}
\label{instability_result}
Denote $\Phi \equiv \Phi(\cdot;a=0)$. There exists $\epsilon > 0$
such that for every sufficiently small $\delta > 0$
there exists $V \in H^1_{\Gamma}$ with $\| V \|_{H^1_{\Gamma}} \leq \delta$
such that the unique solution $\Psi \in C(\mathbb{R},H^1_{\Gamma}) \cap C^1(\mathbb{R},H^{-1}_{\Gamma})$
to the NLS equation (\ref{eq1}) with
the initial datum $\Psi(0,\cdot) = \Phi + V$ satisfies
\begin{equation}
\label{orbital-instab}
\inf_{\theta \in \mathbb{R}} \| e^{-i \theta} \Psi(T,\cdot) - \Phi \|_{H^1(\Gamma)} > \epsilon \quad \mbox{\rm for some \;} T > 0.
\end{equation}
Consequently, the orbit $\{ \Phi e^{i \theta}\}_{\theta \in \mathbb{R}}$
is unstable under the time flow of the NLS equation (\ref{eq1}).
\end{theorem}

\begin{remark}
The result of Theorem \ref{instability_result} is very similar to the instability result in Theorem 2.7 in \cite{KP1}
which was proven for the uniform star graph with $\alpha = 1$.
\end{remark}

Finally, the shifted state $\Phi(\cdot;a)$ with $a < 0$ is a saddle point of the action functional
$\Lambda_{\omega = 1}(\Psi) = E(\Psi) + Q(\Psi)$. The saddle point is known to be spectrally unstable \cite{KP2}.
Consequently, it is also nonlinearly unstable under the time flow of the NLS equation (\ref{eq1})
with fast growing perturbations which break the symmetry (\ref{symmetry}) of the shifted state.

Our numerical results collected together in Section \ref{sec-numerics} illustrate all three stages of the nonlinear
instability of the shifted state with $a > 0$ in the balanced star graph $\Gamma$ with $N = 3$. We show the drift instability for
the shifted states with $a > 0$, the weak instability of the half-soliton state with $a = 0$, and the fast exponential instability
of the shifted states with $a < 0$. We also illustrate numerically that the monotonic increase of the momentum functional
in (\ref{dmom_dt}) can lead to the drift instability even if the assumption $P(\Psi_0) > 0$ of Theorem \ref{main_theorem_1}
on the initial datum $\Psi_0$ is not satisfied.

\section{Linear estimates}
\label{sec-linear}

Recall that the scaling transformation (\ref{scal-transf}) transforms the normalized shifted states $\Phi$ of Lemma \ref{solutions-II}
to the $\omega$-dependent family $\Phi_{\omega}$ of the shifted states. We note the following elementary computations:
\begin{eqnarray}
\label{nonzero-entries-1}
D_1(\omega) = -\langle \Phi_{\omega}(\cdot;a), \partial_{\omega} \Phi_{\omega}(\cdot;a) \rangle_{L^2(\Gamma)} =
-\frac{1}{2} \frac{d}{d\omega} \|\Phi_{\omega} \|_{L^2(\Gamma)}^2 = -\frac{1}{2 \alpha_1^2 \omega^{\frac{1}{2}}}
\end{eqnarray}
and
\begin{eqnarray}
\label{nonzero-entries-2}
D_2(\omega) = -\langle \Phi_{\omega}'(\cdot;a), (\cdot + a) \Phi_{\omega}(\cdot;a) \rangle_{L^2(\Gamma)} =
\frac{1}{2} \|\Phi_{\omega} \|_{L^2(\Gamma)}^2 = \frac{\omega^{\frac{1}{2}}}{\alpha_1^2}.
\end{eqnarray}
We discuss separately the linearization of the shifted state with $a \neq 0$ and the half-soliton state with $a = 0$.

\subsection{Linearization at the shifted state with $a \neq 0$}

For every standing wave solution $\Phi_{\omega}(\cdot;a)$ we define two self-adjoint linear operators
$L_{\pm}(\omega,a) : H^2_{\Gamma} \subset L^2(\Gamma) \to L^2(\Gamma)$
by the differential expressions:
\begin{eqnarray}
\left\{ \begin{array}{l}
L_-(\omega,a) = -\Delta + \omega - 2 \alpha^2 \Phi_{\omega}(\cdot;a)^2, \\
L_+(\omega,a) = -\Delta + \omega - 6 \alpha^2 \Phi_{\omega}(\cdot;a)^2. \end{array} \right.
\label{L_a_omega}
\end{eqnarray}
The operator $L_-(\omega,a)$ acts on the imaginary part of the perturbation to $\Phi_{\omega}(\cdot;a)$ and the operator
$L_+(\omega,a)$ acts on the real part of the perturbation; the latter operator is also the Jacobian operator
for the stationary NLS equation (\ref{eq2}). The spectrum of the self-adjoint operators $L_{\pm}(\omega,a)$
was studied in \cite{KP2}, from which we recall some basic facts.

The continuous spectrum is strictly positive thanks to the fast exponential decay of $\Phi_{\omega}(x;a)$ to zero as $|x| \to \infty$
and Weyl's Theorem:
\begin{equation}
\sigma_c(L_\pm(\omega,a)) = [\omega, \infty),
\end{equation}
where $\omega>0$. The discrete spectrum $\sigma_p(L_\pm(\omega,a)) \subset (-\infty,\omega)$
includes finitely many negative, zero, and positive eigenvalues of finite multiplicities.

\begin{figure}[htbp] 
   \centering
   \includegraphics[width=5in, height = 4in]{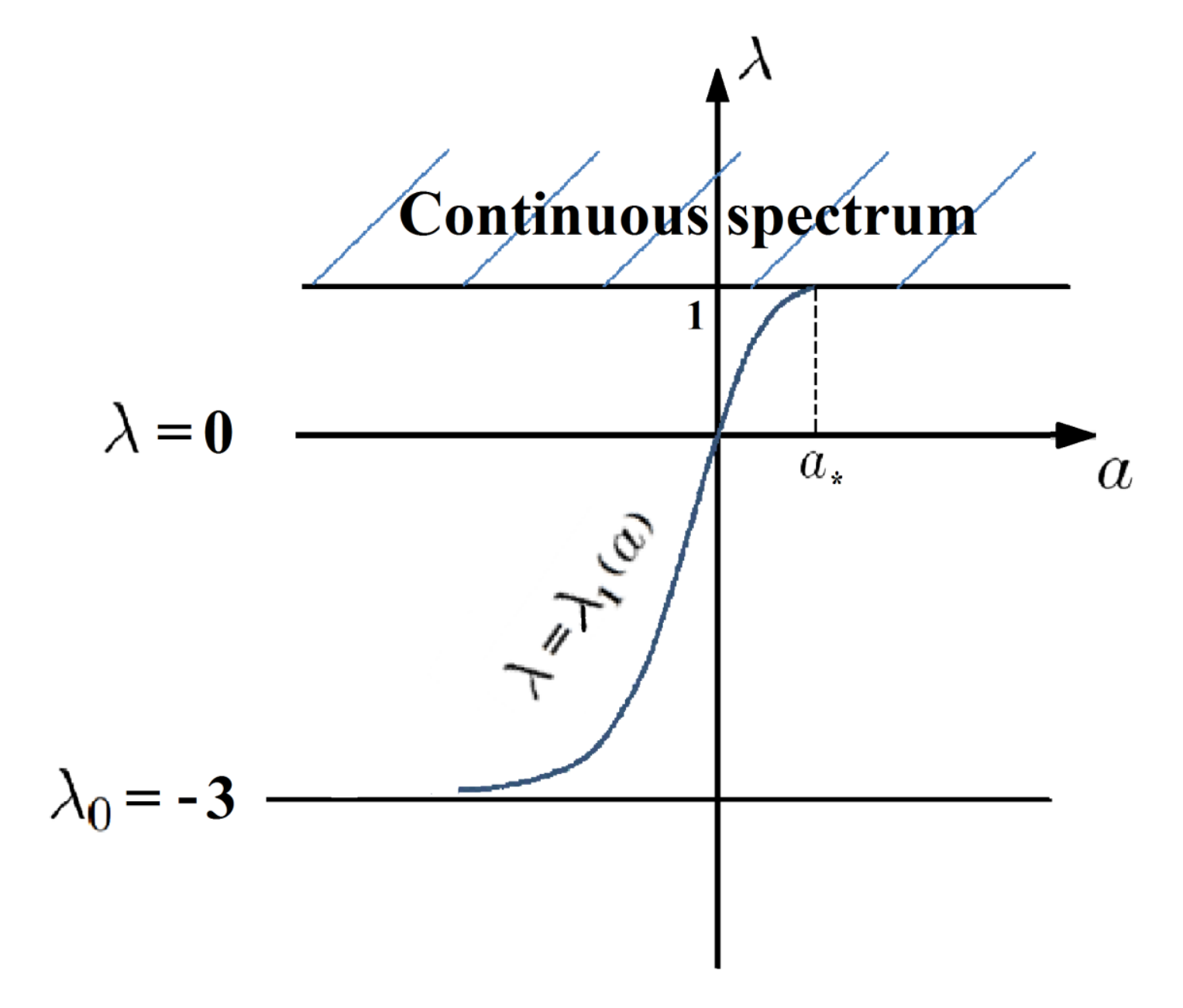}
   \caption{The spectrum of $L_+(\omega,a)$ for $\omega = 1$.
   The continuous spectrum is $[1,\infty)$, while the discrete spectrum is given by
   the eigenvalues $\lambda=0$, $\lambda=-3$, and $\lambda=\lambda_1(a)$ in (\ref{lambda_1_a}).}
   \label{fig-2}
\end{figure}

Eigenvalues of $\sigma_p(L_+(\omega,a)) \subset (-\infty,\omega)$
are known in the explicit form \cite{KP2}. For $\omega = 1$, these eigenvalues are given by:
\begin{itemize}
\item a simple negative eigenvalue $\lambda_0=-3$;
\item a zero eigenvalue $\lambda=0$ which is simple when $a \neq 0$;
\item the additional eigenvalue $\lambda = \lambda_1(a)$ of multiplicity $N-2$ given by
\begin{equation}
\label{lambda_1_a}
\lambda_1(a) = -\frac{3}{2}\tanh(a)\left[ \tanh(a) - \sqrt{1+3\sech(a)} \right].
\end{equation}
It is negative for $a < 0$, zero for $a = 0$, and positive for $a \in (0,a_*)$, where $a_* = \tanh ^{-1}\left(\frac{1}{\sqrt{3}}\right) \approx 0.66$.
The eigenvalue merges into the continuous spectrum as $a \to a_*$.
\end{itemize}
The spectrum of $L_+(\omega,a)$ for $\omega = 1$ is illustrated in Fig. \ref{fig-2}.

Eigenvalues of $\sigma_p(L_-(\omega,a)) \subset (-\infty,\omega)$ are non-negative
and the zero eigenvalue is simple.
If $a \neq 0$, the zero eigenvalues of  $L_+(\omega,a)$ and $L_-(\omega,a)$ are each simple with the eigenvectors given by
\begin{equation}
\label{kernel}
L_+(\omega,a) \Phi_{\omega}'(\cdot;a) = 0, \quad L_-(\omega,a) \Phi_{\omega}(\cdot;a) = 0.
\end{equation}
The eigenvectors in (\ref{kernel}) induce the generalized eigenvectors in
\begin{equation}
\label{kernel-generalized}
L_+(\omega,a) \partial_{\omega} \Phi_{\omega}(\cdot;a)  = - \Phi_{\omega}(\cdot;a), \quad
L_-(\omega,a) (\cdot + a) \Phi_{\omega}(\cdot;a) = -2\Phi_{\omega}'(\cdot;a).
\end{equation}
The following lemma gives coercivity of the quadratic forms associated with the operators
$L_+(\omega,a)$ and $L_-(\omega,a)$ for $a > 0$.

\begin{lemma}
\label{lem-coercivity}
For every $\omega > 0$ and $a > 0$, there exists a positive constant $C(\omega,a)$ such that
\begin{equation}
\label{joint_coercivity}
\langle L_+(\omega,a) U, U \rangle_{L^2(\Gamma)} + \langle L_-(\omega,a) W, W \rangle_{L^2(\Gamma)} \geq C(\omega,a) \|U+iW\|^2_{H^1(\Gamma)}
\end{equation}
if $U$ and $W$ satisfy the orthogonality conditions
\begin{equation}
\label{orthogonal_constraints-linear}
\left\{ \begin{array}{l}
\langle W, \partial_{\omega} \Phi_{\omega}(\cdot;a) \rangle_{L^2(\Gamma)} = 0, \\
\langle U,\Phi_{\omega}(\cdot;a) \rangle_{L^2(\Gamma)} = 0, \\
\langle U, (\cdot + a)\Phi_{\omega}(\cdot;a) \rangle_{L^2(\Gamma)} = 0,  \end{array} \right.
\end{equation}
\end{lemma}

\begin{proof}
The first orthogonality condition in (\ref{orthogonal_constraints-linear}) shifts the lowest (zero) eigenvalue of $L_-(\omega,a)$ to a positive eigenvalue
thanks to the condition (\ref{nonzero-entries-1}) (see Lemma 5.6 in \cite{KP1}) and yields by G{\aa}rding's inequality
the following coercivity bound
$$
\langle L_-(\omega,a) W, W \rangle_{L^2(\Gamma)} \geq C(\omega) \|W\|^2_{H^1(\Gamma)}
$$
independently of $a$. The second orthogonality condition in (\ref{orthogonal_constraints-linear}) shifts the lowest (negative)
eigenvalue of $L_-(\omega,a)$ to a positive eigenvalue thanks to the same condition
(\ref{nonzero-entries-1}) (see Lemma 3.8 in \cite{KP1}) and yields
$$
\langle L_+(\omega,a) U, U \rangle_{L^2(\Gamma)} \geq 0
$$
with $\langle L_+(\omega,a) U, U \rangle_{L^2(\Gamma)} = 0$ if and only if $U$ is  proportional
to $\Phi_{\omega}'(\cdot;a).$ The zero eigenvalue of $L_+(\omega,a)$ is preserved by the constraint since
$$
\langle \Phi_{\omega}(\cdot;a), \Phi_{\omega}'(\cdot;a) \rangle_{L^2(\Gamma)}  = 0.
$$
Finally, the third orthogonality condition in (\ref{orthogonal_constraints-linear}) shifts the zero eigenvalue of $L_+(\omega,a)$ to a
positive eigenvalue thanks to the condition (\ref{nonzero-entries-2}). By G{\aa}rding's inequality,
this yields the coercivity bound
$$
\langle L_+(\omega,a) U, U \rangle_{L^2(\Gamma)} \geq C(\omega,a) \|U\|^2_{H^1(\Gamma)},
$$
where $C(\omega,a)$ depends on $a$ because the gap between the zero eigenvalue and the rest of the
positive spectrum in $\sigma_p(L_+(\omega,a))$ exists for $a > 0$ but vanishes as $a \to 0$.
\end{proof}

\begin{remark}
\label{rem-coercivity}
For every $\omega > 0$, the positive constant $C(\omega,a)$ in (\ref{joint_coercivity}) satisfies
$$
C(\omega,a) \to 0 \quad \mbox{\rm as} \quad a \to 0.
$$
This is because the zero eigenvalue in $\sigma_p(L_+(\omega,a = 0))$ has multiplicity $(N-1)$ and
the $(N-2)$ eigenvectors of $L_+(\omega,a=0)$ satisfy the last two
orthogonality conditions (\ref{orthogonal_constraints-linear}) as is seen from the proof
of Lemma \ref{orth_conditions}.
\end{remark}

\begin{remark}
\label{rem-negative-eig}
For $a < 0$, the result of Lemma \ref{lem-coercivity} is invalid because $\sigma_p(L_+(\omega,a))$
includes another negative eigenvalue as is seen from Fig. \ref{fig-2}.
\end{remark}

\begin{remark}
The orthogonality conditions in (\ref{orthogonal_constraints-linear}) are typically referred to as
the symplectic orthogonality conditions, because they express orthogonality of residual terms
$U$ and $W$ for real and imaginary parts of the perturbation to $\Phi_{\omega}(\cdot;a)$
to the eigenvectors and generalized eigenvectors of the spectral stability problem
expressed by $L_+(\omega,a)$ and $L_-(\omega,a)$ and the symplectic structure of the NLS equation.
Note that our approach will not rely on adding one more orthogonality constraint
$\langle W, \Phi_{\omega}'(\cdot;a) \rangle_{L^2(\Gamma)} = 0$. As a result, we will only use
three parameters for modulations of the stationary state orbit
$\{e^{i\theta} \Phi_\omega (\cdot;a)\}_{\theta \in \mathbb{R}, a \in \mathbb{R}^+, \omega \in \mathbb{R}^+}$.
\end{remark}

\subsection{Linearization at the half-soliton state}

For $a = 0$, we denote operators $L_{\pm}(\omega) \equiv L_{\pm}(\omega,a=0)$.
The kernel of the operator $L_+(\omega)$ is spanned by an
orthogonal basis consisting of $N-1$ eigenvectors, which we denote by
$\{U^{(1)}_\omega, U^{(2)}_\omega, \cdots, U^{(N-1)}_\omega\}$.
The following lemma specifies properties of these basis eigenvectors.

\begin{lemma}
\label{kernel_elements}
There exists an orthogonal basis $\{U^{(1)}_\omega, U^{(2)}_\omega, \cdots, U^{(N-1)}_\omega\}$ of the kernel of $L_+(\omega)$
satisfying the orthogonality condition
\begin{equation}
\label{L_c_2}
\langle U, \Phi_\omega \rangle_{L^2(\Gamma)} = 0.
\end{equation}
The eigenvectors can be represented in the following way: for $j = 1$,
\begin{equation}
\label{U_0}
U^{(1)}_\omega := (\alpha_1^{-1}\phi_\omega',\alpha_2^{-1}\phi_\omega', \dots, \alpha_N^{-1}\phi_\omega'),
\end{equation}
and for $j = 2,\dots,N-1$,
\begin{equation}
\label{U_j}
U^{(j)}_\omega : = (\underbrace{ 0, \dots, 0}_\text{\rm (j-1) elements},
r_j \phi_\omega', \alpha_{j+1}^{-1} \phi_\omega', \dots, \alpha_{N}^{-1} \phi_\omega'), \quad
r_j = -\left( \sum_{i=j+1}^N \frac{1}{\alpha_i^2} \right) \alpha_j,
\end{equation}
where $\phi_\omega(x) = \omega^{\frac{1}{2}} \sech (\omega^{\frac{1}{2}}x)$, $x \in \mathbb{R}$.
\end{lemma}

\begin{proof}
Let $U = (u_1,u_2, \dots, u_N) \in H^2_\Gamma$ be an eigenvector for the zero eigenvalue
of the operator $L_+(\omega)$. Each component of the eigenvalue problem $L_+(\omega)U =0$
satisfies
\begin{equation}
\label{ode-u-j}
-u_j''(x)+\omega u_j(x) - 6\omega \sech^2(\sqrt{\omega} x)u_j(x) = 0,
\end{equation}
where $x \in \mathbb{R}^-$ on the first edge and $x \in \mathbb{R}^+$ on the remaining edges.
Since $H^2(\mathbb{R}^\pm)$ are continuously embedded into $C^1(\mathbb{R}^\pm)$,
if $U \in H^2(\Gamma)$, then both $u_j(x)$ and $u_j'(x)$ decay to zero as $|x| \to \infty$.
Such solutions to the differential equations (\ref{ode-u-j}) are given uniquely
by $u_j(x) = a_j\phi'_\omega(x)$ up to multiplication by a constant $a_j$.
Therefore, the eigenvector $U$ is given by
\begin{equation}
\label{Q_function}
U = (a_1 \phi_\omega', a_2 \phi_\omega', \dots, a_N \phi_\omega').
\end{equation}
The eigenvector $U \in H^2_\Gamma$ must satisfy the boundary conditions in (\ref{H2}).
The continuity conditions hold since $\phi_\omega'(0) = 0$, whereas the Kirchhoff condition implies
\begin{equation}
\label{c_constraint}
\frac{a_1}{\alpha_1} = \sum_{j=2}^N \frac{a_j}{\alpha_j}.
\end{equation}
Since the scalar equation (\ref{c_constraint}) relates $N$ unknowns, the space of solutions
for $(a_1, a_2, \dots, a_N)$ is $(N-1)$-dimensional and the kernel of the operator $L_+(\omega)$
is $(N-1)$-dimensional. Let $\{U^{(1)}_\omega, U^{(2)}_\omega, \dots, U^{(N-1)}_\omega\}$ be an orthogonal basis of the kernel,
which can be constructed from any set of basis vectors by applying the Gram-Schmidt orthogonalization process.

Direct computations show that if $U$ is given by (\ref{Q_function}), then
$$
\langle  U, \Phi_\omega \rangle_{L^2(\Gamma)} =
\left(\sum_{j=2}^N \frac{a_j}{\alpha_j} - \frac{a_1}{\alpha_1} \right)
\langle \phi_\omega', \phi_\omega \rangle_{L^2(\mathbb{R}^+)},
$$
which means that the condition (\ref{c_constraint}) is equivalent to
$\langle  U, \Phi_\omega \rangle_{L^2(\Gamma)} = 0$. Therefore, all elements
in the orthogonal basis satisfy the orthogonality condition (\ref{L_c_2}).

It remains to prove that the orthogonal basis can be characterized in the form given in (\ref{U_0})--(\ref{U_j}).
From the constraint (\ref{constraint}), we can take $a_j = \alpha_j^{-1}$ for all $j$ in (\ref{c_constraint})
to set the first eigenvector $U^{(1)}_\omega$ to be defined by (\ref{U_0}).
The last eigenvector $U^{(N-1)}_\omega$ can be defined by
\begin{equation}
\label{U_2}
U^{(N-1)}_\omega : = (0, \dots, 0, r_{N-1} \phi_\omega', \alpha_N^{-1}\phi_\omega'),
\end{equation}
where $r_{N-1}$ is defined to satisfy the orthogonality condition $\langle U^{(1)}_\omega, U^{(N-1)}_\omega \rangle_{L^2(\Gamma)} = 0$
and the condition (\ref{c_constraint}). In fact, both conditions are equivalent since the first $(N-2)$ entries of $U^{(N-1)}_\omega$ are zero and
$$
\langle U^{(1)}_\omega, U^{(N-1)}_\omega \rangle_{L^2(\Gamma)} =
\|\phi_\omega'\|^2_{L^2(\mathbb{R}^+)}
\left( \frac{r_{N-1}}{\alpha_{N-1}} + \frac{1}{\alpha^2_{N}} \right)
$$
with $\|\phi_\omega'\|^2_{L^2(\mathbb{R}^+)}\neq 0$. Hence $r_{N-1}$ is defined by
$$
r_{N-1} = -\frac{\alpha_{N-1}}{\alpha_N^2}.
$$
The remaining eigenvectors $U^{(j)}_\omega$ in (\ref{U_j}) are constructed recursively
from $j = N-2$ to $j = 2$. By direct computations we obtain that the orthogonality condition
$\langle U^{(1)}_\omega, U^{(j)}_\omega \rangle_{L^2(\Gamma)} = 0$ is equivalent to the constraint (\ref{c_constraint}).
Moreover, all the eigenvectors are mutually orthogonal thanks to the recursive construction of $U_\omega^{(j)}$,
\end{proof}

We denote the eigenspace for the kernel of $L_+(\omega)$ by
\begin{equation}
\label{X_c}
X_{\omega}:={\rm span}\{U^{(1)}_\omega, U^{(2)}_\omega, \cdots, U^{(N-1)}_\omega\}.
\end{equation}
For each $j = 1,2,\dots,N-1$, we construct the generalized eigenvector $W^{(j)}_\omega \in H^2_{\Gamma}$
by solving
$$
L_-(\omega) W^{(j)}_\omega = U^{(j)}_\omega,
$$
which exists thanks to the orthogonality condition (\ref{L_c_2})
since $\Phi_{\omega}$ spans the kernel of $L_-(\omega)$.
Explicitly, representing $U_\omega^{(j)}$ from (\ref{U_0})--(\ref{U_j}) by
\begin{equation}
\label{U-2}
U_\omega^{(j)} = \phi_\omega' e_j
\end{equation}
with some $x$-independent vectors $e_j \in \mathbb{R}^N$, we get for the same vectors $e_j$
\begin{equation}
\label{W_vector}
W_\omega^{(j)} = \chi_{\omega} e_j,
\end{equation}
where $\chi_{\omega}(x) = -\frac{1}{2}x\phi_\omega(x)$, $x \in \mathbb{R}$.
We denote the eigenspace for the generalized kernel of $L_-(\omega)$ by
\begin{equation}
\label{X_c_star}
X^*_{\omega} :={\rm span}\{W^{(1)}_\omega, W^{(2)}_\omega, \cdots, W^{(N-1)}_\omega\},
\end{equation}
Similarly to Lemma 5.4 in \cite{KP1}, the following lemma gives coercivity of the quadratic forms
associated with the operators $L_+(\omega)$ and $L_-(\omega)$.

\begin{lemma}
\label{orth_conditions}
For every $\omega > 0$, there exists a positive constant $C(\omega)$ such that
\begin{equation}
\label{joint_coercivity-a-zero}
\langle L_+(\omega) U, U \rangle_{L^2(\Gamma)} + \langle L_-(\omega) W, W \rangle_{L^2(\Gamma)} \geq C(\omega) \|U+iW\|^2_{H^1(\Gamma)}
\end{equation}
if $U \in X_{\omega}^*$ and $W \in X_{\omega}$ satisfying the additional orthogonality conditions
\begin{equation}
\label{orthogonal_constraints-linear-a-zero}
\left\{ \begin{array}{l}
\langle W, \partial_{\omega} \Phi_{\omega} \rangle_{L^2(\Gamma)} = 0, \\
\langle U,\Phi_{\omega} \rangle_{L^2(\Gamma)} = 0.  \end{array} \right.
\end{equation}
\end{lemma}

\begin{proof}
We claim that basis vectors in $X_{\omega}$ and $X_{\omega}^*$ satisfy
the following orthogonality conditions:
\begin{itemize}
\item $\{\langle U_\omega^{(j)}, U_\omega^{(k)} \rangle_{L^2(\Gamma)} \}_{1\leq j,k \leq N-1}$
is a positive diagonal matrix;
\item $\{\langle W_\omega^{(j)}, W_\omega^{(k)} \rangle_{L^2(\Gamma)} \}_{1\leq j,k \leq N-1}$
is a positive diagonal matrix;
\item $\{\langle U_\omega^{(j)}, W_\omega^{(k)} \rangle_{L^2(\Gamma)} \}_{1\leq j,k \leq N-1}$
is a positive diagonal matrix.
\end{itemize}
Indeed, orthogonality of $\{ U_\omega^{(1)},\dots, U_{\omega}^{(N-1)}\}$ is established by
Lemma \ref{kernel_elements}. Therefore, the vectors $\{ e_1, \dots, e_{N-1}\}$ in (\ref{U-2}) are
orthogonal in $\mathbb{R}^{N-1}$. Orthogonality of $\{ W_\omega^{(1)},\dots, W_{\omega}^{(N-1)}\}$
follows by the explicit representation (\ref{W_vector}) due to orthogonality
of the vectors $\{ e_1, \dots, e_{N-1}\}$ in $\mathbb{R}^{N-1}$. The sets
$\{ U_\omega^{(1)},\dots, U_{\omega}^{(N-1)}\}$ and
$\{ W_\omega^{(1)},\dots, W_{\omega}^{(N-1)}\}$ are mutually orthogonal by the same reason.
Finally, we have for every $j=1,\dots, N$
\begin{equation}
\label{positive-elements}
\langle U_\omega^{(j)}, W_\omega^{(j)} \rangle_{L^2(\Gamma)} =
\frac{\alpha_j^2}{4} \left(\sum_{i=j}^N \frac{1}{\alpha_i^2} \right)
\left(\sum_{i=j+1}^N \frac{1}{\alpha_i^2} \right)
\|\phi_\omega\|^2_{L^2(\mathbb{R}^+)} (> 0).
\end{equation}
The rest of the proof is similar to the proof of Lemma \ref{lem-coercivity} with the only difference that
the third orthogonality condition (\ref{orthogonal_constraints-linear}) is replaced by the $(N-1)$
orthogonality conditions in $U \in X_{\omega}^*$. The constraint $U \in X_{\omega}^*$
provide the shift of the zero eigenvalue
of $L_+(\omega)$ of algebraic multiplicity $(N-1)$ to positive eigenvalues thanks to
the condition that $\{\langle U_\omega^{(j)}, W_\omega^{(k)} \rangle_{L^2(\Gamma)} \}_{1\leq j,k \leq N-1}$
is a positive diagonal matrix.
\end{proof}

\section{Drift of the shifted states with $a>0$}
\label{sec-theorem-1}

The proof of Theorem \ref{main_theorem_1} is divided into several steps.
First, we decompose a unique global solution $\Psi$ to the NLS equation (\ref{eq1}) into the modulated stationary state
$\{e^{i\theta} \Phi_\omega (\cdot;a)\}_{\theta \in \mathbb{R}, a \in \mathbb{R}, \omega \in \mathbb{R}^+}$
and the symplectically orthogonal remainder terms. Second, we estimate the rate of change of the modulation parameter
$a(t)$ in time $t$ and show that $a'(t) < 0$ for $t > 0$. Third, we use energy estimates to control the time evolution
of the modulation parameter $\omega(t)$ and the remainder terms. Although the decomposition works for any $a(t)$,
we only consider $a(t) > 0$.

\subsection{Step 1: Symplectically orthogonal decomposition}

Any point in $H^1_{\Gamma}$ close to an orbit $\{e^{i\theta} \Phi (\cdot;a_0)\}_{\theta \in \mathbb{R}}$
for some $a_0 \in \mathbb{R}$ can be represented by a superposition of
a point on the family $\{e^{i\theta} \Phi_{\omega} (\cdot;a)\}_{\theta \in \mathbb{R}, a \in \mathbb{R}, \omega \in \mathbb{R}^+}$
and a symplectically orthogonal remainder term. Here and in what follows, we denote $\Phi \equiv \Phi_{\omega = 1}$.
The following lemma provides details of this symplectically orthogonal decomposition.

\begin{lemma}
\label{unique_decomposition}
Fix $a_0 \in \mathbb{R}$. There exists some $\delta_0>0$ such that for every $\Psi \in H^1_\Gamma$ satisfying
\begin{equation}
\label{orth-given}
\delta := \inf_{\theta \in \mathbb{R}} \| \Psi - e^{i \theta} \Phi(\cdot; a_0) \|_{H^1(\Gamma)} \leq \delta_0,
\end{equation}
there exists a unique choice for real-valued $(\theta,\omega,a) \in \mathbb{R} \times \mathbb{R}^+ \times \mathbb{R}$ and real-valued
$(U,W) \in H^1_{\Gamma} \times H^1_{\Gamma}$ in the decomposition
\begin{equation}
\label{orth-decomposition}
\Psi(x) = e^{i \theta} \left[ \Phi_{\omega}(x;a) + U(x) + i W(x) \right],
\end{equation}
subject to the orthogonality conditions
\begin{equation}
\label{orthogonal_constraints}
\left\{ \begin{array}{l}
\langle W, \partial_{\omega} \Phi_{\omega}(\cdot;a) \rangle_{L^2(\Gamma)} = 0, \\
\langle U,\Phi_{\omega}(\cdot;a) \rangle_{L^2(\Gamma)} = 0, \\
\langle U, (\cdot + a)\Phi_{\omega}(\cdot;a) \rangle_{L^2(\Gamma)} = 0,  \end{array} \right.
\end{equation}
where $\omega$, $a$, and $(U,W)$ satisfy the estimate
\begin{equation}
\label{orth-bound}
| \omega - 1| + |a-a_0| + \| U + iW \|_{H^1(\Gamma)} \leq C \delta,
\end{equation}
for some positive constant $C > 0$. Moreover, the map
from $\Psi \in H^1_\Gamma$ to $(\theta,\omega,a) \in \mathbb{R} \times \mathbb{R}^+ \times \mathbb{R}$ and $(U,W) \in H^1_{\Gamma} \times H^1_{\Gamma}$
is $C^{\omega}$.
\end{lemma}

\begin{proof}
Define the following vector function $G(\theta,\omega, a;\Psi) : \mathbb{R} \times \mathbb{R}^+ \times \mathbb{R}
\times H^1_{\Gamma} \mapsto \mathbb{R}^3$ given by
$$
G(\theta,\omega, a;\Psi) := \left[ \begin{array}{l}
 \langle {\rm Im}(\Psi - e^{i \theta} \Phi_{\omega}(\cdot;a)), \partial_{\omega} \Phi_{\omega}(\cdot;a) \rangle_{L^2(\Gamma)} \\
 \langle {\rm Re}(\Psi - e^{i \theta} \Phi_{\omega}(\cdot;a)), \Phi_{\omega}(\cdot;a) \rangle_{L^2(\Gamma)} \\
 \langle {\rm Re}(\Psi - e^{i \theta} \Phi_{\omega}(\cdot;a)), (\cdot + a)\Phi_{\omega}(\cdot;a) \rangle_{L^2(\Gamma)}
 \end{array} \right],
$$
the zeros of which represent the orthogonality constraints in (\ref{orthogonal_constraints}).
The function $G(\theta,\omega, a;\Psi)$ is $C^{\omega}$ with respect to its arguments.

Let $\theta_0$ be the argument of
$\inf_{\theta \in \mathbb{R}} \| \Psi - e^{i \theta} \Phi(\cdot;a_0) \|_{H^1(\Gamma)}$
for a given $\Psi \in H^1_{\Gamma}$. The vector function $G$ is a $C^{\omega}$ map from $\mathbb{R} \times \mathbb{R}^+ \times \mathbb{R}$
to $\mathbb{R}^3$
since the map $\mathbb{R}^+ \times \mathbb{R} \ni (\omega,a) \mapsto \Phi_{\omega}(\cdot;a) \in L^2(\Gamma)$ is $C^{\omega}$ in both variables.
Moreover, if $\Psi \in H^1_{\Gamma}$ satisfies (\ref{orth-given}), then
\begin{equation}
\label{bound-on-G}
\| G(\theta_0,1, a_0;\Psi) \|_{\mathbb{R}^3} \leq C \delta
\end{equation}
for a $\delta$-independent constant $C > 0$. Also we have
\begin{equation*}
D_{(\theta,\omega, a)} G(\theta_0,1, a_0;\Psi)  = D + B,
\end{equation*}
where $D = {\rm diag}(d_1,d_1,d_2)$ with entries $d_1 \equiv D_1(\omega = 1)$
and $d_2 \equiv D_2(\omega = 1)$ given by (\ref{nonzero-entries-1}) and (\ref{nonzero-entries-2}),
whereas $B$ is a matrix satisfying the estimate $\| B \|_{\mathbb{M}_{3 \times 3}} \leq C \delta$ for a $\delta$-independent constant $C>0$.
Since $d_1, d_2 \neq 0$, the matrix $D$ is invertible and there exists $\delta_0>0$ such that
the Jacobian $D_{(\theta,\omega, a)} G(\theta_0,1, a_0;\Psi)$
is invertible for every $\delta \in (0, \delta_0)$ with the bound
\begin{equation}
\label{bound-on-D-G}
\| [D_{(\theta,\omega, a)} G(\theta_0,1, a_0;\Psi)]^{-1} \|_{\mathbb{M}_{3 \times 3}} \leq C
\end{equation}
for a $\delta$-independent constant $C > 0$.
By the local inverse mapping theorem, for the given $\Psi \in H^1_\Gamma$ satisfying (\ref{orth-given}),
the equation $G(\theta, \omega, a; \Psi) = 0$ has a unique solution $(\theta, \omega, a) \in \mathbb{R}^3$
in a neighborhood of the point $(\theta_0, 1, a_0)$. Since $G(\theta,\omega, a;\Psi)$ is $C^{\omega}$
with respect to its arguments, the solution $(\theta, \omega, a) \in \mathbb{R} \times \mathbb{R}^+ \times \mathbb{R}$ is $C^{\omega}$
with respect to $\Psi \in H^1_{\Gamma}$. The Taylor expansion of $G(\theta, \omega, a; \Psi) = 0$
around $(\theta_0,1,a_0)$,
$$
0=G(\theta_0, 1, a_0; \Psi)  + D_{(\theta,\omega, a)} G(\theta_0,1, a_0;\Psi) (\theta-\theta_0, \omega-1, a-a_0)^T +
\mathcal{O}(|\theta - \theta_0|^2 + |\omega - 1|^2 + |a - a_0|^2),
$$
together with the bounds (\ref{bound-on-G}) and (\ref{bound-on-D-G})
implies the bound (\ref{orth-bound}) for $|\omega-1|$ and $|a-a_0|$.
From the decomposition (\ref{orth-decomposition}) and with use of the triangle inequality for $(\theta,\omega, a)$
near $(\theta_0,1, a_0)$, it follows
that $(U,W)$ are uniquely defined in $H^1_{\Gamma}$ and satisfy the bound in (\ref{orth-bound}).
In addition, $(U,W) \in H^1_{\Gamma}$ are $C^{\omega}$
with respect to $\Psi \in H^1_{\Gamma}$.
\end{proof}

If $\delta > 0$ in the initial bound (\ref{theorem_bound_delta}) is sufficiently small,
we can represent the initial datum $\Psi_0 \in H^1_\Gamma$ to the Cauchy problem
associated with the NLS equation (\ref{eq1}) in the form:
\begin{equation}
\label{initial_datum}
\Psi_0(x) = \Phi(x;a_0) + U_0(x) + i W_0(x), \quad \| U_0  + iW_0 \|_{H^1(\Gamma)}\leq \delta,
\end{equation}
subject to the orthogonality conditions
\begin{equation}
\label{initial_constraints}
\left\{ \begin{array}{l}
\langle W_0, \partial_{\omega} \Phi_{\omega}|_{\omega=1}(\cdot;a_0) \rangle_{L^2(\Gamma)} = 0, \\
\langle U_0,\Phi(\cdot;a_0) \rangle_{L^2(\Gamma)} = 0, \\
\langle U_0, (\cdot + a_0)\Phi(\cdot;a_0) \rangle_{L^2(\Gamma)} = 0.
\end{array} \right.
\end{equation}
By Lemma \ref{unique_decomposition}, the orthogonal decomposition (\ref{initial_datum})
with (\ref{initial_constraints}) implies that $\theta(0) = 0$ and $\omega(0) = 1$ initially.
Although this is not the most general case for the initial datum satisfying (\ref{theorem_bound_delta}),
this simplification is used to illustrate the proof of Theorem \ref{main_theorem_1}.
A generalization for initial datum $\Psi_0 \in H^1_{\Gamma}$ with $\theta(0) \neq 0$ and $\omega(0) \neq 1$ is straightforward.

By the global well-posedness theory \cite{AdamiJDE1,KP2},
the NLS equation (\ref{eq1}) with the initial datum $\Psi_0 \in H^1_{\Gamma}$
generates a unique solution $\Psi \in C(\mathbb{R}, H^1_\Gamma) \cap C^1(\mathbb{R}, H^{-1}_\Gamma)$.
By continuous dependence of the solution on the initial datum and by Lemma \ref{unique_decomposition},
for every $\epsilon \in (0, \delta_0)$ with $\delta_0$ in the bound (\ref{orth-given})
there exists $t_0>0$ such that the unique solution $\Psi$ satisfies
\begin{equation}
\label{delta_0_bound-sec4}
\inf_{\theta \in \mathbb{R}} \| e^{-i \theta} \Psi(t,\cdot) - \Phi \|_{H^1(\Gamma)} \leq \epsilon, \quad t\in [0,t_0]
\end{equation}
and can be uniquely decomposed in the form:
\begin{equation}
\label{sol_orthogonal_decom}
\Psi(t,x) = e^{i \theta (t)}\left[ \Phi_{\omega(t)}(x;a(t)) + U(t,x) + i W(t,x) \right],
\end{equation}
subject to the orthogonality conditions
\begin{equation}
\label{orth-decomposition-time-constraints}
\left\{ \begin{array}{l}
\langle W(t,\cdot), \partial_{\omega} \Phi_{\omega} |_{\omega = \omega(t)} (\cdot;a(t))\rangle_{L^2(\Gamma)} =0, \\
\langle U(t,\cdot), \Phi_{\omega(t)}(\cdot;a(t)) \rangle_{L^2(\Gamma)} = 0, \\
\langle U(t,\cdot), (\cdot + a(t))\Phi_{\omega}(\cdot;a(t)) \rangle_{L^2(\Gamma)} = 0.\end{array} \right.
\end{equation}
By the smoothness of the map in Lemma \ref{unique_decomposition} and
by the well-posedness of the time flow of the NLS equation (\ref{eq1}), we have
$U,W \in C([0,t_0], H^1_\Gamma) \cap C^1([0,t_0], H^{-1}_\Gamma)$
and $(\theta,\omega,a) \in C^1([0,t_0],\mathbb{R} \times \mathbb{R}^+ \times \mathbb{R})$.

In order to prove Theorem \ref{main_theorem_1}, we control $\omega(t)$, $U(t,\cdot)$, and $W(t,\cdot)$
from energy estimates and $a(t)$ from modulation equations, whereas $\theta(t)$ plays
no role in the bound (\ref{theorem_bound_epsilon}). Note that the modulation of $a(t)$
captures the irreversible drift of the shifted states along the incoming edge
towards the vertex of the balanced star graph.
We would not see this drift without using the parameter $a(t)$ and we would not be
able to control $\omega(t)$, $U(t,\cdot)$, and $W(t,\cdot)$ from energy estimates
without the third constraint in (\ref{orth-decomposition-time-constraints}) because
of the zero eigenvalue of $L_+(\omega,a)$, see Lemma \ref{lem-coercivity}.

\subsection{Step 2: Monotonicity of $a(t)$}

We use the orthogonal decomposition (\ref{sol_orthogonal_decom}) with (\ref{orth-decomposition-time-constraints})
in order to obtain the evolution system for the remainder terms $(U,W)$ and for the modulation parameters $(\theta,\omega,a)$.
By analyzing the modulation equation for $a(t)$, we relate the rate of change of $a(t)$
and the value of the momentum functional $P(\Psi)$ given by (\ref{momentum}).

\begin{lemma}
\label{lem-monotonicity}
Assume that the unique solution $\Psi \in C([0,t_0], H^1_\Gamma) \cap C^1([0,t_0], H^{-1}_\Gamma)$
represented by (\ref{sol_orthogonal_decom}) and (\ref{orth-decomposition-time-constraints})
satisfies
\begin{equation}
\label{apriori-bound}
|\omega(t) - 1 | + \| U(t,\cdot) + i W(t,\cdot) \|_{H^1(\Gamma)} \leq \epsilon, \quad t \in [0,t_0]
\end{equation}
with $\epsilon \in (0,\delta_0)$ and $\delta_0$ defined in (\ref{orth-given}).
The time evolution of the translation parameter $a(t)$ is given by
\begin{equation}
\label{change_of_a}
\dot{a}(t) = -\alpha_1^2 \omega^{-\frac{1}{2}} P(\Psi) \left[ 1 + \mathcal{O}(\|U+iW\|_{H^1(\Gamma)}) \right]
+ \mathcal{O}(\|U+iW\|^2_{H^1(\Gamma)}),
\end{equation}
where $P(\Psi)$ is given by (\ref{momentum}).
\end{lemma}

\begin{proof}
By substituting (\ref{sol_orthogonal_decom}) into the NLS equation (\ref{eq1})
and by using the rotational and translation symmetries, we obtain the time evolution
system for the remainder terms:
\begin{eqnarray}
\nonumber
\frac{d}{dt} \begin{pmatrix} U \\ W \end{pmatrix} & = &
\begin{pmatrix} 0 & L_-(\omega,a) \\ -L_+(\omega,a) & 0 \end{pmatrix} \begin{pmatrix} U \\ W \end{pmatrix}
+ (\dot{\theta} - \omega) \begin{pmatrix} W \\ -(\Phi_{\omega} + U) \end{pmatrix} \\
& \phantom{t} &
- \dot{\omega} \begin{pmatrix} \partial_{\omega} \Phi_{\omega} \\ 0 \end{pmatrix}
- \dot{a} \begin{pmatrix} \Phi_{\omega}' \\ 0 \end{pmatrix}
+ \begin{pmatrix} -R_U \\ R_W \end{pmatrix},
\label{time-evolution}
\end{eqnarray}
where $\Phi_{\omega} \equiv \Phi_{\omega}(x;a)$, the prime denotes derivative in $x$,
the dot denotes derivative in $t$, the linearized operators are given by (\ref{L_a_omega}), and the residual terms are given by
\begin{eqnarray}
\left\{ \begin{array}{l}
R_U =  2 \alpha^2 \left( 2 \Phi_{\omega} U + U^2 + W^2 \right) W, \\
R_W =  2 \alpha^2 \left[\Phi_{\omega} ( 3 U^2 + W^2) + (U^2 + W^2) U \right].\end{array} \right.
\label{R_U_W_terms}
\end{eqnarray}
By using the orthogonality conditions (\ref{orth-decomposition-time-constraints}),
we obtain the modulation equations for parameters $(\theta,\omega, a)$ from the system (\ref{time-evolution}):
{\small \begin{equation}
\label{time-evolution-ode}
A \left[ \begin{matrix} \dot{\theta} - \omega \\ \dot{\omega} \\ \dot{a} \end{matrix} \right]
= \left[ \begin{matrix} 0 \\ 0 \\
- 2\langle \Phi_{\omega}'(\cdot;a), W \rangle_{L^2(\Gamma)} \end{matrix} \right] +
\left[ \begin{matrix} \langle \Phi_{\omega}(\cdot;a), R_U \rangle_{L^2(\Gamma)} \\ \langle \partial_{\omega} \Phi_{\omega}, R_W \rangle_{L^2(\Gamma)} \\
- \langle (\cdot + a) \Phi_{\omega}(\cdot;a), R_W \rangle_{L^2(\Gamma)} \end{matrix} \right],
\end{equation}}where the matrix $A$ is given by
{\small $$
A = \left[ \begin{matrix} \langle \Phi_{\omega}(\cdot;a), W \rangle_{L^2(\Gamma)} &
-\langle \partial_{\omega} \Phi_{\omega}(\cdot;a), \Phi_{\omega}(\cdot;a) - U \rangle_{L^2(\Gamma)} &
\langle \Phi_{\omega}'(\cdot;a), U \rangle_{L^2(\Gamma)} \\
\langle \partial_{\omega} \Phi_{\omega}(\cdot;a), \Phi_{\omega}(\cdot;a) + U  \rangle_{L^2(\Gamma)} &
-\langle \partial_{\omega}^2 \Phi_{\omega}(\cdot;a), W \rangle_{L^2(\Gamma)} &
-\langle \partial_\omega \Phi_{\omega}'(\cdot;a), W \rangle_{L^2(\Gamma)} \\
-\langle (\cdot + a) \Phi_\omega(\cdot;a), W  \rangle_{L^2(\Gamma)} &
-\langle (\cdot + a) \partial_{\omega} \Phi_\omega(\cdot;a), U \rangle_{L^2(\Gamma)} &
\langle (\cdot + a) \Phi_\omega(\cdot;a)', \Phi_{\omega}(\cdot;a) - U \rangle_{L^2(\Gamma)}
 \end{matrix} \right].
$$}If $(U,W) = (0,0)$, the matrix $A$ is invertible since
$$
A_0 = \left[ \begin{matrix}0 & D_1(\omega) & 0 \\
-D_1(\omega) & 0 & 0 \\
0 & 0 & - D_2(\omega)  \end{matrix} \right].
$$
has nonzero elements thanks to (\ref{nonzero-entries-1}) and (\ref{nonzero-entries-2}). Therefore,
under the assumption (\ref{apriori-bound}) with small $\epsilon > 0$,
we have
\begin{equation}
\label{bound-on-inverse-A}
\| A^{-1} \|_{\mathbb{M}_{3 \times 3}} \leq C
\end{equation}
for an $\epsilon$-independent constant $C > 0$. This bound implies that
the time-evolution of the translation parameter $a(t)$ is given by
\begin{equation}
\label{change_of_a-again}
\dot{a} = \frac{2\langle \Phi_{\omega}'(\cdot;a), W \rangle_{L^2(\Gamma)}}{D_1(\omega)}
\left[ 1 + \mathcal{O}(\|U+iW\|_{H^1(\Gamma)}) \right] + \mathcal{O}(\|U+iW\|^2_{H^1(\Gamma)}).
\end{equation}
On the other hand, the momentum functional $P(\Psi)$ in (\ref{momentum})
can be computed at the solution $\Psi$ in the orthogonal decomposition (\ref{sol_orthogonal_decom})
as follows
\begin{eqnarray}
\nonumber
P(\Psi) & = & \langle \Phi_{\omega}(\cdot;a), W' \rangle_{L^2(\Gamma)}
- \langle \Phi_{\omega}'(\cdot;a), W \rangle_{L^2(\Gamma)} + \mathcal{O}(\|U+iW\|^2_{H^1(\Gamma)}) \\
& = & - 2 \langle \Phi_{\omega}'(\cdot;a), W \rangle_{L^2(\Gamma)} + \mathcal{O}(\|U+iW\|^2_{H^1(\Gamma)}),
\label{momentum_decomposition}
\end{eqnarray}
where the integration by parts does not result in any contribution from the vertex at $x = 0$
thanks to the boundary conditions in (\ref{H2}) and the constraint (\ref{constraint}).
Combining (\ref{change_of_a-again}) and (\ref{momentum_decomposition}) with the exact computation (\ref{nonzero-entries-2})
yields expansion (\ref{change_of_a}).
\end{proof}

\begin{corollary}
\label{corollary-a}
In addition to (\ref{apriori-bound}), assume that $\Psi_0$ in (\ref{initial_datum}) is chosen
such that $P(\Psi_0)>0$. There exists $\epsilon_0$ sufficiently small such that for every $\epsilon \in (0,\epsilon_0)$
there exists $\delta > 0$ such that the map $t \mapsto a(t)$ is strictly decreasing for $t \in [0,t_0]$.
\end{corollary}

\begin{proof}
The map $t \mapsto P(\Psi)$ is monotonically increasing as it can be seen from the expression (\ref{dmom_dt}).
Therefore, if the initial datum $\Psi_0$ in (\ref{initial_datum}) satisfies $P(\Psi_0)>0$, then
\begin{equation}
\label{1help}
P(\Psi) \geq P(\Psi_0) > 0 \quad {\rm for \; all} \;\; t \in [0,t_0].
\end{equation}
It follows from (\ref{initial_datum}) and (\ref{momentum_decomposition})
that there are $\delta$-independent constants $C_-, C_+ > 0$ such that
\begin{equation}
\label{2help}
C_- \delta \leq P(\Psi_0) \leq C_+ \delta.
\end{equation}
Then, it follows from (\ref{apriori-bound}),
(\ref{change_of_a}), (\ref{1help}) and (\ref{2help})
that there exist $\delta$ and $\epsilon$-independent constants $C_1, C_2 > 0$ such that
$$
-\dot{a} \geq C_1 \delta - C_2 \epsilon^2.
$$
If $\delta$ satisfies $\delta \geq C \epsilon^2$ for a given small $\epsilon > 0$
with an $\epsilon$-independent constant $C > C_1^{-1} C_2$
then $-\dot{a} \geq (C_1 C - C_2) \epsilon^2 > 0$ so that
the map $t \mapsto a(t)$ is strictly decreasing for $t \in [0,t_0]$.
\end{proof}

\subsection{Step 3: Energy estimates}

The coercivity bound (\ref{joint_coercivity}) in Lemma \ref{lem-coercivity} allows us to control the time evolution
of $\omega(t)$, $U(t,\cdot)$, and $W(t,\cdot)$, as long as $a(t)$ is bounded away from zero.
The following result provides this control from energy estimates.

\begin{lemma}
\label{bound_global}
Let $\Psi$ be the unique global solution to the NLS equation (\ref{eq1}) given by
(\ref{sol_orthogonal_decom})--(\ref{orth-decomposition-time-constraints}) for $t \in [0,t_0]$
with some $t_0 > 0$ such that the initial data $\Psi(0,\cdot) = \Psi_0$ satisfies (\ref{initial_datum})--(\ref{initial_constraints}).
Assume that $a(t) \geq \bar{a}$ for $t \in [0,t_0]$. For every $\bar{a} > 0$, there exists a $\delta$-independent
positive constant $K(\bar{a})$ such that
\begin{equation}
\label{bound-on-U-W}
|\omega(t)-1|^2 + \|U(t,\cdot) + i W(t,\cdot) \|_{H^1(\Gamma)}^2 \leq K(\bar{a}) \delta^2, \quad t \in [0,t_0].
\end{equation}
\end{lemma}

\begin{proof}
Recall that the shifted state $\Phi_{\omega}(\cdot;a)$ is a critical point of
the action functional $\Lambda_{\omega}(\Psi) = E(\Psi) + \omega Q(\Psi)$ in (\ref{lyapunov-funct}).
By using the decomposition (\ref{sol_orthogonal_decom}) and the rotational invariance of the
NLS equation (\ref{eq1}), we define the following energy function:
\begin{equation}
\label{energy-difference}
\Delta(t) := E(\Phi_{\omega(t)} + U(t,\cdot) + i W(t,\cdot)) - E(\Phi)
+ \omega(t) \left[ Q(\Phi_{\omega(t)} + U(t,\cdot) + i W(t,\cdot)) - Q(\Phi) \right].
\end{equation}
Expanding $\Delta$ into Taylor series, we obtain
\begin{equation}
\label{energy-expansion}
\Delta = D(\omega) + \langle L_+(\omega,a) U, U \rangle_{L^2(\Gamma)} + \langle L_-(\omega,a) W, W \rangle_{L^2(\Gamma)} + N_\omega(U, W),
\end{equation}
where $N_\omega(U,W) = {\rm O}(\| U +i W \|_{H^1(\Gamma)}^3)$ and $D(\omega)$ is defined by
$$
D(\omega) := E(\Phi_{\omega}) - E(\Phi) + \omega \left[ Q(\Phi_{\omega})  - Q(\Phi) \right].
$$
Since $D'(\omega) = Q(\Phi_{\omega}) - Q(\Phi)$
thanks to the variational problem for the standing wave $\Phi_{\omega}$, we have $D(1) = D'(1) = 0$, and
\begin{equation}
\label{energy-0}
D(\omega) = (\omega - 1)^2 \langle \Phi, \partial_{\omega} \Phi_{\omega} |_{\omega = 1} \rangle_{L^2(\Omega)}
+ \mathcal{O}(|\omega - 1|^3).
\end{equation}
It follows from the initial decomposition (\ref{initial_datum})--(\ref{initial_constraints}) that
\begin{equation}
\label{energy-1}
\Delta(0) = E(\Phi + U_0 + i W_0) - E(\Phi) + Q(\Phi + U_0 + i W_0) - Q(\Phi)
\end{equation}
satisfies the bound
\begin{equation}
\label{bound-on-Delta}
|\Delta(0)| \leq C_0 \delta^2
\end{equation}
for a $\delta$-independent constant $C_0 > 0$.
On the other hand, the energy and mass conservation in (\ref{energy}) imply that
\begin{equation}
\label{energy-2}
\Delta(t) = \Delta(0) + \left( \omega(t) - 1 \right) \left[ Q(\Phi + U_0 + i W_0) - Q(\Phi)\right],
\end{equation}
where the remainder term also satisfies
\begin{equation}
\label{bound-on-Q}
|Q(\Phi + U_0 + i W_0) - Q(\Phi)| \leq C_0 \delta^2
\end{equation}
for a $\delta$-independent constant $C_0 > 0$.
The representation (\ref{energy-2}) together with the expression (\ref{energy-expansion})
allows us to control $\omega(t)$ near $\omega(0) = 1$ and the remainder terms
$(U,W)$ in $H^1_\Gamma$ as follows:
\begin{eqnarray}
\nonumber
\Delta(0) & = & (\omega - 1)^2 \langle \Phi, \partial_{\omega} \Phi_{\omega} |_{\omega = 1} \rangle_{L^2(\Omega)}
- \left( \omega - 1 \right) \left[ Q(\Phi + U_0 + i W_0) - Q(\Phi)\right] \\
& \phantom{t} & + \langle L_+(\omega,a) U, U \rangle_{L^2(\Gamma)} + \langle L_-(\omega,a) W, W
\rangle_{L^2(\Gamma)} + \mathcal{O}(|\omega - 1|^3 + \| U +i W \|_{H^1(\Gamma)}^3).
\label{energy-3}
\end{eqnarray}
By using the expansion (\ref{energy-3}), the coercivity bound (\ref{joint_coercivity}),
and the bounds (\ref{bound-on-Delta}) and (\ref{bound-on-Q}), we obtain
$$
C_0 \delta^2 \geq \Delta(0) \geq \frac{1}{2\alpha_1^2 |\omega|^{\frac{1}{2}}} (\omega - 1)^2
- C_0 \delta^2 |\omega - 1| + C(\omega,a) \| U + i W \|^2_{H^1(\Gamma)} + \mathcal{O}(|\omega - 1|^3 + \| U + i W \|_{H^1(\Gamma)}^3),
$$
from which the bound (\ref{bound-on-U-W}) follows.
\end{proof}

\begin{remark}
\label{rem-K}
By Remark \ref{rem-coercivity}, for every $\omega > 0$, we have $C(\omega,a) \to 0$ as $a \to 0$.
Therefore, we have $K(\bar{a}) \to \infty$ as $\bar{a} \to 0$.
\end{remark}

\subsection{Proof of Theorem \ref{main_theorem_1}}

The initial datum satisfies the initial decomposition (\ref{initial_datum})--(\ref{initial_constraints}) with small $\delta$
and initial conditions $\theta(0) = 0$, $\omega(0) = 1$, and $a(0) = a_0$ with $a_0 > 0$.
Thanks to the continuous dependence of the solution of the NLS equation (\ref{eq1}) on initial datum, the solution is represented by
the orthogonal decomposition (\ref{sol_orthogonal_decom})--(\ref{orth-decomposition-time-constraints})
on a short time interval $[0,t_0]$ for some $t_0 > 0$. Hence, it satisfies the apriori bound (\ref{delta_0_bound-sec4}).
The modulation parameters $\theta(t)$, $\omega(t)$, and $a(t)$
are defined for $t \in [0,t_0]$ and $a(t) \geq \bar{a}$ for some $\bar{a} > 0$ for $t \in [0,t_0]$.
By energy estimates in Lemma \ref{bound_global}, the parameter $\omega(t)$ and
the remainder terms $(U,W) \in H^1_{\Gamma}$ satisfy the bound (\ref{bound-on-U-W})
with a $\delta$-independent positive constant $K(\bar{a})$. For given small $\epsilon > 0$ and
$\nu > 0$ in Theorem \ref{main_theorem_1}, let us define
\begin{equation}
\label{def-on-delta}
K_{\nu} := \max_{\bar{a} \in [\nu,a_0]} K(\bar{a}), \quad \delta := K_{\nu}^{-\frac{1}{2}} \epsilon.
\end{equation}
Then, the bound (\ref{bound-on-U-W}) provides the bound (\ref{apriori-bound}) of Lemma \ref{lem-monotonicity}
for all $t \in [0,t_0]$.
Assume that the initial datum also satisfies $P(\Psi_0) > 0$. By Corollary \ref{corollary-a},
the map $t \mapsto a$ is strictly decreasing for $t \in [0,t_0]$ if $\delta$ satisfies
$\delta \geq C \epsilon^2$ for a $\delta$ and $\epsilon$-independent
constant $C > 0$. The definition of $\delta$ in (\ref{def-on-delta}) is compatible with the latter bound
if $\epsilon \in (0,\epsilon_0)$ with
$$
\epsilon_0 := \frac{1}{C \sqrt{K_{\nu}}}.
$$
If in addition $\epsilon_0 \leq \delta_0$, where $\delta_0$ is defined by (\ref{orth-given}) in Lemma
\ref{unique_decomposition}, then the solution $\Psi(t,\cdot) \in H^1_{\Gamma}$ for $t \in [0,t_0]$ satisfies
the conditions of Lemma \ref{unique_decomposition} so that
the orthogonal decomposition (\ref{sol_orthogonal_decom}) with (\ref{orth-decomposition-time-constraints}) is
continued beyond the short time interval $[0,t_0]$ to the maximal time interval $[0,T]$ as long as
$a(t) \geq \nu$ for $t \in [0,T]$. Thanks to the monotonicity argument in Lemma \ref{lem-monotonicity}
and Corollary \ref{corollary-a}, for every $\epsilon \in (0,\epsilon_0)$,
there exists a finite $T > 0$ such that $\lim_{t \to T} a(t) = \nu$. Note that $T = \mathcal{O}(\epsilon^{-2})$
as $\epsilon \to 0$. Theorem \ref{main_theorem_1} is proved.

\begin{remark}
It follows that $K_{\nu} \to \infty$ as $\nu \to 0$ by Remark \ref{rem-K}
so that $\epsilon_0$ may be sufficiently small for a fixed small $\nu > 0$.
\end{remark}

\section{Instability of the half-soliton state}
\label{sec-theorem-2}

The proof of Theorem \ref{instability_result} is divided into several steps.
First, we decompose a unique global solution $\Psi$ to the NLS equation (\ref{eq1}) into the modulated stationary state
$\{e^{i\theta} \Phi_\omega \}_{\theta \in \mathbb{R}, \omega \in \mathbb{R}^+}$
and the symplectically orthogonal remainder terms, where $\Phi_{\omega} \equiv \Phi_{\omega}(\cdot;a=0)$.
Next, we provide a secondary decomposition of the remainder terms as a superposition of projections
to the bases in $X_{\omega}$ and $X_{\omega}^*$ in (\ref{X_c}) and (\ref{X_c_star})
and the symplectically orthogonally remainder terms. Energy estimates are used to
control the time evolution of $\omega(t)$ and the remainder terms. Finally,
we consider the perturbed Hamiltonian system for projections to the bases
in $X_{\omega}$ and $X_{\omega}^*$ and prove that the truncated system is unstable near the zero equilibrium point.
This instability drives the instability of the solution $\Psi$ under the time flow
of the NLS equation (\ref{eq1}) away from the half-soliton state
$\{e^{i\theta} \Phi_\omega \}_{\theta \in \mathbb{R}, \omega \in \mathbb{R}^+}$.

\subsection{Step 1: Primary decomposition}

For every sufficiently small $\delta > 0$,
we consider the initial datum $\Psi_0 \in H^1_\Gamma$ to the Cauchy problem
associated with the NLS equation (\ref{eq1}) in the form:
\begin{equation}
\label{a_0_initial_datum}
\Psi_0 = \Phi+U_0+iW_0, \quad
\|U_0+iW_0\|_{H^1(\Gamma)} \leq \delta,
\end{equation}
subject to the orthogonality conditions
\begin{equation}
\label{initial_constraints-a-0}
\left\{ \begin{array}{l}
\langle W_0, \partial_{\omega} \Phi_{\omega}|_{\omega=1} \rangle_{L^2(\Gamma)} = 0, \\
\langle U_0,\Phi \rangle_{L^2(\Gamma)} = 0,
\end{array} \right.
\end{equation}
where $\Phi \equiv \Phi(\cdot;a=0)$. By the global well-posedness theory \cite{AdamiJDE1,KP2}, this initial datum
generates a unique global solution $\Psi \in C(\mathbb{R},H^1_{\Gamma}) \cap C^1(\mathbb{R},H^{-1}_{\Gamma})$
to the NLS equation (\ref{eq1}).

We use the decomposition of the unique global solution
$\Psi \in C(\mathbb{R},H^1_{\Gamma}) \cap C^1(\mathbb{R},H^{-1}_{\Gamma})$
into the modulated stationary state $\{e^{i\theta}\Phi_\omega\}$ with $\omega$ close to $\omega_0=1$ and
the symplectically orthogonal remainder terms. This decomposition is similar to Lemma \ref{unique_decomposition}
with $a_0 = 0$ with the only change that $a = 0$ is set in the decomposition (\ref{orth-decomposition})
and the remainder terms satisfy the first two of the three orthogonality conditions
in (\ref{orthogonal_constraints}).

By continuity of the global solution and by Lemma \ref{unique_decomposition},
for every $\epsilon \in (0, \delta_0)$ with $\delta_0$ in the bound (\ref{orth-given})
there exists $t_0>0$ such that the unique solution $\Psi$ satisfies
\begin{equation}
\label{delta_0_bound}
\inf_{\theta \in \mathbb{R}} \| e^{-i \theta} \Psi(t,\cdot) - \Phi \|_{H^1(\Gamma)} \leq \epsilon, \quad t\in [0,t_0]
\end{equation}
and can be uniquely represented as
\begin{equation}
\label{a_0_orth_decomposition}
\Psi(t,x) = e^{i \theta(t)} \left[ \Phi_{\omega(t)}(x) + U(t,x) + i W(t,x) \right],
\end{equation}
subject to the orthogonality conditions
\begin{equation}
\label{a_0_orth_decomposition_constraints}
\left\{ \begin{array}{l}
\langle W(t,\cdot), \partial_{\omega} \Phi_{\omega} |_{\omega = \omega(t)} \rangle_{L^2(\Gamma)} = 0, \\
\langle U(t,\cdot), \Phi_{\omega(t)} \rangle_{L^2(\Gamma)} = 0. \end{array} \right.
\end{equation}
Since $\Psi \in C(\mathbb{R},H^1_{\Gamma}) \cap C^1(\mathbb{R},H^{-1}_{\Gamma})$
and the map $\mathbb{R} \ni \omega \mapsto \Phi_{\omega} \in H^1_{\Gamma}$ is smooth, we obtain
$(\theta,\omega) \in C^1([0,t_0],\mathbb{R} \times \mathbb{R}^+)$, hence
$U,W \in C([0,t_0],H^1_{\Gamma}) \cap C^1([0,t_0],H^{-1}_{\Gamma})$.

The choice (\ref{a_0_initial_datum}) with (\ref{initial_constraints-a-0})
implies that $\omega(0) = 1$ and $\theta(0) = 0$. Although this is again special,
it is nevertheless sufficient for the proof of the instability result.
In order to prove Theorem \ref{instability_result} we fix $\epsilon \in (0,\delta_0)$ and
we intend to show that there exists such $T>0$ that the bound (\ref{delta_0_bound})
is satisfied for all $t \in [0,T)$, but fails to satisfy for $t \geq T$.

We substitute the decomposition (\ref{a_0_orth_decomposition}) into the NLS equation (\ref{eq1})
to get the time evolution system for the remainder terms $(U,W)$. The following
lemma reports the estimates on the evolution of the modulation parameters $\theta(t)$ and $\omega(t)$.

\begin{lemma}
\label{lem-parameters}
Assume that the unique solution $\Psi \in C([0,t_0], H^1_\Gamma) \cap C^1([0,t_0], H^{-1}_\Gamma)$
represented by (\ref{a_0_orth_decomposition}) and (\ref{a_0_orth_decomposition_constraints}) satisfies
\begin{equation}
\label{bounds-parameters-apriori}
|\omega(t) - 1 | + \| U(t,\cdot) + i W(t,\cdot) \|_{H^1(\Gamma)} \leq \epsilon, \quad t \in [0,t_0]
\end{equation}
with $\epsilon \in (0,\delta_0)$ and $\delta_0$ defined in (\ref{orth-given}).
There exists an $\epsilon$-independent constant $A > 0$ such that for every $t \in [0,t_0]$,
\begin{equation}
\label{bounds-parameters}
\left\{ \begin{array}{l} | \dot{\theta}(t) - \omega(t) | \leq A \left( \| U(t,\cdot) \|^2_{H^1(\Gamma)} + \| W(t,\cdot) \|^2_{H^1(\Gamma)} \right), \\
| \dot{\omega}(t) | \leq A \| U(t,\cdot) \|_{H^1(\Gamma)} \| W(t,\cdot) \|_{H^1(\Gamma)}. \end{array} \right.
\end{equation}
\end{lemma}

\begin{proof}
The time evolution system for the remainder terms is the same as the system (\ref{time-evolution}) but
with $a(t) = 0$ and $\dot{a}(t) = 0$ for $t \in [0,t_0]$. Using the orthogonality conditions
(\ref{a_0_orth_decomposition_constraints}) yields the same system of modulation equations
as in (\ref{time-evolution-ode}) but constrained by the first two equations and a $2\times 2$ coefficient matrix
denoted by ${\bf A}$. The estimates (\ref{bounds-parameters}) are obtained from invertibility of ${\bf A}$
satisfying $\| {\bf A}^{-1} \|_{\mathbb{M}_{2\times 2}} \leq C$ for an $\epsilon$-independent constant $C > 0$
and the explicit form $R_U$ and $R_W$ in (\ref{R_U_W_terms}).
\end{proof}

\subsection{Step 2: Secondary decomposition}

Recall the eigenspaces $X_{\omega}$ and $X_{\omega}^*$  in (\ref{X_c}) and (\ref{X_c_star}).
We decompose the remainder terms $(U,W)$ in (\ref{a_0_orth_decomposition}) as follows:
\begin{equation}\label{a_0_r1}
U(t,x) = \sum_{j=1}^{N-1} c_j(t) U_{\omega(t)}^{(j)}(x) + U^{\perp}(t,x), \quad
W(t,x) = \sum_{j=1}^{N-1} b_j(t) W_{\omega(t)}^{(j)}(x) + W^{\perp}(t,x),
\end{equation}
subject to the orthogonality conditions for $U^{\perp}(t,\cdot) \in X_{\omega(t)}^*$ and $W^{\perp}(t,\cdot) \in X_{\omega(t)}$.
The coefficients $c = (c_1,c_2,\dots,c_{N-1}) \in \mathbb{R}^{N-1}$, $b = (b_1,b_2,\dots,b_{N-1}) \in \mathbb{R}^{N-1}$
and the remainder terms $(U^{\perp},W^{\perp})$ in (\ref{a_0_r1}) are uniquely determined for each $(U,W)$
due to the mutual orthogonality of the basis vectors used in Lemma \ref{orth_conditions}.
We also have  $c,b \in C^1([0,t_0],\mathbb{R}^{N-1})$ and
$U^{\perp}, W^{\perp} \in C([0,t_0],H^1_{\Gamma}) \cap C^1([0,t_0],H^{-1}_{\Gamma})$ since $\omega \in C^1([0,t_0],\mathbb{R})$ and
$U, W \in C([0,t_0],H^1_{\Gamma}) \cap C^1([0,t_0],H^{-1}_{\Gamma})$.

\begin{remark}
\label{remark-a}
Thanks to the explicit form (\ref{U_0}), the parameter $c_1(t)$ in the decomposition
(\ref{a_0_orth_decomposition}) and (\ref{a_0_r1}) plays the same role
as the parameter $a(t)$ in the decomposition (\ref{sol_orthogonal_decom}), whereas $a(0) = a_0$
in the initial decomposition (\ref{initial_datum}) is set to $a_0 = 0$.
\end{remark}

By substituting (\ref{a_0_r1}) into the evolution equation for the remainder terms
and using orthogonality conditions for the new remainder terms $U^{\perp}(t,\cdot) \in X_{\omega(t)}^*$
and $W^{\perp}(t,\cdot) \in X_{\omega(t)}$,
we obtain the time-evolution system for the coefficients in the form:
\begin{eqnarray}
\label{a_0_system-U-W}
\left\{ \begin{array}{l}
\langle W_{\omega}^{(j)},U^{(j)}_{\omega}  \rangle_{L^2(\Gamma)} \left( \frac{dc_j}{dt} - b_j \right) = R_c^{(j)}, \\
\langle W_{\omega}^{(j)},U^{(j)}_{\omega}  \rangle_{L^2(\Gamma)} \frac{db_j}{dt} = R_b^{(j)}, \end{array} \right.
\end{eqnarray}
with
{\small \begin{eqnarray*}
R_c^{(j)} & = &  \dot{\omega} \left[ \langle \partial_{\omega} W_{\omega}^{(j)}, U^{\perp} \rangle_{L^2(\Gamma)}
-  \sum_{i=1}^{N-1} c_i \langle W_{\omega}^{(j)}, \partial_{\omega} U^{(i)}_{\omega}  \rangle_{L^2(\Gamma)} \right]
+ (\dot{\theta}-\omega) \langle W_{\omega}^{(j)}, W \rangle_{L^2(\Gamma)}
- \langle W_{\omega}^{(j)}, R_U \rangle_{L^2(\Gamma)}, \\
R_b^{(j)} & = &  \dot{\omega} \left[ \langle \partial_{\omega} U_{\omega}^{(j)}, W^{\perp} \rangle_{L^2(\Gamma)}
-  \sum_{i=1}^{N-1} b_i \langle U_{\omega}^{(j)}, \partial_{\omega} W^{(i)}_{\omega}  \rangle_{L^2(\Gamma)} \right]
- (\dot{\theta}-\omega) \langle U_{\omega}^{(j)}, U \rangle_{L^2(\Gamma)}
+ \langle U_{\omega}^{(j)}, R_W \rangle_{L^2(\Gamma)},
\end{eqnarray*}}where the terms $R_U$ and $R_W$ are given by (\ref{R_U_W_terms}) and the orthogonality conditions
$$
\langle U^{(j)}_\omega, \Phi_\omega \rangle_{L^2(\Gamma)} =
\langle W^{(j)}_\omega, \partial_\omega \Phi_\omega \rangle_{L^2(\Gamma)} = 0, \quad 1\leq j \leq N-1
$$
have been used. The time-evolution system (\ref{a_0_system-U-W}) will be truncated and studied in Step 3, whereas
$\omega(t)$ and the remainder terms $(U^\perp,W^\perp)$ are controlled as small
perturbations by using the energy estimates. The following lemma gives the estimates on these terms.

\begin{lemma}
\label{lem-remainder-last}
Assume that there exists a positive constant $A$ such that for every $\epsilon > 0$ and some $t_0 > 0$,
the remainder terms of the solution $\Psi$ decomposed as (\ref{a_0_orth_decomposition}) and
(\ref{a_0_r1}) satisfy
\begin{equation}
\label{remainder_terms_bound}
|\omega(t)-1| + \|c(t)\| + \|b(t)\| + \|U^\perp(t,\cdot) + i W^\perp(t,\cdot) \|_{H^1(\Gamma)} \leq A\epsilon, \quad t \in [0,t_0].
\end{equation}
There exist an $\epsilon$-independent constant $C > 0$ such that
\begin{equation}
\label{bound-omega-U-W}
|\omega(t) - 1|^2 +  \| U^{\perp}(t,\cdot)+i W^{\perp}(t,\cdot) \|_{H^1(\Gamma)}^2 \leq C
 \left( \delta^2 + \|b\|^2 + \|c\|^3 \right), \quad t \in [0,t_0],
\end{equation}
where $\delta$ is given in (\ref{a_0_initial_datum}) for the initial datum $\Psi_0$.
\end{lemma}

\begin{proof}
The expansion of the energy function (\ref{energy-difference}) with the help of the explicit expressions in (\ref{energy}) implies that the term
$N_\omega(U, W)$ in (\ref{energy-expansion}) can be written as
\begin{equation}
\label{N-expansion}
N_\omega(U,W) = -4 \langle \alpha^2 \Phi U, (U^2 + W^2) \rangle_{L^2(\Gamma)} +
\mathcal{O}(\|U+iW\|^4_{H^1(\Gamma)}).
\end{equation}
Substituting the secondary decomposition (\ref{a_0_r1}) into the energy function (\ref{energy-expansion})
with the use of (\ref{N-expansion}) yields
\begin{eqnarray}
\nonumber
\Delta & = & D(\omega) + \langle L_+(\omega) U^{\perp}, U^{\perp} \rangle_{L^2(\Gamma)} + \langle L_-(\omega) W^{\perp}, W^{\perp}
\rangle_{L^2(\Gamma)} + 2 H_0(c,b) \\
\label{energy-expansion-second_0}
& \phantom{t} & \phantom{text} + \widetilde{N}(\omega,c,b,U^{\perp},W^{\perp}),
\end{eqnarray}
where $H_0(c,b)$ is defined by
\begin{equation}
\label{Hamiltonian}
H_0(c,b) = \frac{1}{2} \sum_{j=1}^{N-1} \langle W^{(j)},U^{(j)}  \rangle_{L^2(\Gamma)} b_j^2 -
2 \sum_{j=1}^{N-1} \sum_{k=1}^{N-1} \sum_{n=1}^{N-1}
\langle \alpha^2 \Phi U^{(j)}, U^{(k)} U^{(n)} \rangle_{L^2(\Gamma)} c_j c_k c_n,
\end{equation}
and $\widetilde{N}(\omega,c,b,U^{\perp},W^{\perp})$ is bounded by
\begin{eqnarray*}
|\widetilde{N}(\omega,c,b,U^{\perp},W^{\perp})| & \leq &
A_1 \left( \| c \|^2 \| U^{\perp} \|_{H^1(\Gamma)} +
\| U^{\perp} \|_{H^1(\Gamma)}^3 + \| c \|^4 + \| c \| \| b \|^2
 + \| c \| \| W^{\perp} \|_{H^1(\Gamma)}^2 \right. \\
& \phantom{t} & \left.
+ \| b \|^2 \| U^{\perp} \|_{H^1(\Gamma)} +  \| U^{\perp} \|_{H^1(\Gamma)} \| W^{\perp} \|_{H^1(\Gamma)}^2 +
|\omega-1| \|b\|^2  + |\omega-1| \|c\|^3\right)
\end{eqnarray*}
with $\epsilon$-independent positive constant $A_1$.
The expansion (\ref{energy-expansion-second_0}) also holds due to Banach algebra property of $H^1(\Gamma)$
and the assumption (\ref{remainder_terms_bound}). Combining the representations of $\Delta$
given by (\ref{energy-2}) and (\ref{energy-expansion-second_0}), we get
\begin{eqnarray}
\nonumber
\Delta(0) - 2 H_0(c,b) & = & D(\omega) - \left( \omega - 1 \right) \left[ Q(\Phi + U_0 + i W_0) - Q(\Phi)\right] \\
& \phantom{t} & + \langle L_+(\omega) U^{\perp}, U^{\perp} \rangle_{L^2(\Gamma)} + \langle L_-(\omega) W^{\perp}, W^{\perp}
\rangle_{L^2(\Gamma)} + \widetilde{N}(\omega, c,b,U^{\perp},W^{\perp}).
\nonumber
\end{eqnarray}
Bounds (\ref{bound-on-Delta}) and (\ref{bound-on-Q}) imply that there exists a $\delta$-independent constant $A_2$ such that
$$
|\Delta(0)| + |Q(\Phi + U_0 + i W_0) - Q(\Phi)| \leq A_2 \delta^2.
$$
Moreover, it can be seen directly from (\ref{Hamiltonian}) that
$|H_0(c,b)| \leq A_3 (\|c\|^3 + \|b\|^2)$ for some generic positive constant $A_3$.
By using the same estimates as in the proof of Lemma \ref{bound_global},
a priori assumption (\ref{remainder_terms_bound}) together with
the coercivity bound (\ref{joint_coercivity-a-zero}) in Lemma \ref{orth_conditions}
imply that there exists an $\epsilon$-independent constant $C$ in the bound (\ref{bound-omega-U-W}).
\end{proof}

\subsection{Step 3: The reduced Hamiltonian system}

Estimates (\ref{bounds-parameters}) in Lemma \ref{lem-parameters}
and (\ref{bound-omega-U-W}) in Lemma \ref{lem-remainder-last}, as well as the representation of $(R_U, R_W)$
in (\ref{R_U_W_terms}) imply that the time-evolution system (\ref{a_0_system-U-W}) is
a perturbation of the following Hamiltonian system of degree $N-1$:
\begin{eqnarray}
\label{normal-form-time} \phantom{texttext}
\left\{ \begin{array}{l}  \langle W^{(j)},U^{(j)} \rangle_{L^2(\Gamma)} \frac{d \gamma_j}{dt} = \frac{\partial H_0}{\partial \beta_j}, \\
\langle W^{(j)},U^{(j)} \rangle_{L^2(\Gamma)} \frac{d \beta_j}{d t} = - \frac{\partial H_0}{\partial \gamma_j},
\end{array} \right.
\end{eqnarray}
where $H_0(\gamma,\beta)$ is the Hamiltonian given by (\ref{Hamiltonian}).
Direct computation with the help of the representations
(\ref{U_0}) and (\ref{U_j}) in Lemma \ref{kernel_elements} imply that if $j \geq k >n$, then
$$
\langle \alpha^2 \Phi U^{(j)}, U^{(k)} U^{(n)} \rangle_{L^2(\Gamma)} =
\left( \frac{r_k}{\alpha_k} + \sum_{i=k+1}^N \frac{1}{\alpha_i^2} \right) \int_0^\infty \phi (\phi')^3 dx = 0
$$
due to the explicit formula for $r_k$ in (\ref{U_j}). Therefore,
one can rewrite the representation (\ref{Hamiltonian}) for $H_0(\gamma,\beta)$ in the explicit form:
\begin{equation}
\label{energy-cube}
H_0(\gamma,\beta) = \frac{1}{2} \sum_{j=1}^{N-1} M_j b_j^2 - 2 \sum_{j=2}^{N-1} R_j \gamma_j^3 -
6 \sum_{j=1}^{N-1} \sum_{k>j}^{N-1} P_k \gamma_j \gamma_k^2,
\end{equation}
where
\begin{eqnarray*}
M_j & := & \langle W^{(j)},U^{(j)}  \rangle_{L^2(\Gamma)}, \\
R_j & := & \langle \alpha^2 \Phi U^{(j)}, U^{(j)} U^{(j)} \rangle_{L^2(\Gamma)}, \\
P_k & := & \langle \alpha^2 \Phi U^{(j)}, U^{(k)} U^{(k)} \rangle_{L^2(\Gamma)}.
\end{eqnarray*}
Note that the coefficient $P_k$ is independent of $j$ if $k > j$. Thanks to the construction
of the eigenvectors in (\ref{U_0}) and (\ref{U_j}), the explicit expressions for coefficients $R_j$ and $P_k$
are given by
\begin{equation}
\label{coefficient-R}
R_j = \alpha_j^4 \left( \sum_{i=j}^N \frac{1}{\alpha_i^2} \right)
\left( \sum_{i=j+1}^N \frac{1}{\alpha_i^2} \right)
\left( \frac{1}{\alpha_j^2} - \sum_{i=j+1}^N \frac{1}{\alpha_i^2} \right)
\int_0^\infty \phi (\phi')^3 dx,
\end{equation}
and
\begin{equation}
\label{coefficient-P}
P_k = \alpha_k^2 \left( \sum_{i=k}^N \frac{1}{\alpha_i^2} \right)
\left( \sum_{i=k+1}^N \frac{1}{\alpha_i^2} \right)
\int_0^\infty \phi (\phi')^3 dx.
\end{equation}
It follows from (\ref{positive-elements}) and (\ref{coefficient-P}) that
$M_j > 0$ and $P_k < 0$ since $\phi (\phi')^3 < 0$ on $\mathbb{R}^+$.
Also it follows from (\ref{constraint}) and (\ref{coefficient-R}) that $R_1 = 0$.

The following lemma states that the zero equilibrium point is nonlinearly unstable
in the reduced system (\ref{normal-form-time}) with the Hamiltonian (\ref{energy-cube}).

\begin{lemma}
\label{Ham_instability}
There exists $\epsilon>0$ such that for every sufficiently small
$\delta > 0$, there is an initial point $(\gamma(0),\beta(0)) \in \mathbb{R}^{N-1} \times \mathbb{R}^{N-1}$ with
$\| \gamma(0)\| + \| \beta(0) \| \leq \delta$ such that
the unique solution of the reduced Hamiltonian system (\ref{normal-form-time})
with (\ref{energy-cube}) satisfies for some $t_0 > 0$:
$\| \gamma(t_0) \| = \epsilon$ and $\| \gamma(t) \| > \epsilon$ for $t > t_0$.
Moreover, if $\epsilon > 0$ is small then $t_0 = \mathcal{O}(\epsilon^{-1/2})$,
$\gamma(t) = \mathcal{O}(\epsilon)$, and
$\beta(t) = \mathcal{O}(\epsilon^{3/2})$ for all $t \in [0,t_0]$.
\end{lemma}

\begin{proof}
First, we claim that there exists an invariant subspace of solutions of
the reduced Hamiltonian system (\ref{normal-form-time}) with (\ref{energy-cube}) given by
\begin{equation}
\label{reduction}
S := \{ \gamma_1 = C \gamma_2,
\gamma_3 =\gamma_4 =  \dots = \gamma_{N-1} =0\}
\end{equation}
for some constant $C \neq 0$. Indeed, eliminating $\beta_j$, we close
the reduced system (\ref{normal-form-time}) on $\gamma_j$ for every $j=1, \dots, N-1$:
\begin{equation}
\label{second-order-system}
M_j \frac{d^2 \gamma_j}{d t^2} = 6 R_j \gamma_j^2 + 12 \sum_{i=1}^{j-1} P_j \gamma_i \gamma_j + 6 \sum_{k=j+1}^{N-1} P_k \gamma_k^2.
\end{equation}
It follows directly that $\gamma_3 = \gamma_4 = \dots = \gamma_{N-1} =0$ is an invariant solution
of the last $(N-3)$ equations of system (\ref{second-order-system}).
Since $R_1 = 0$ from (\ref{constraint}) and (\ref{coefficient-R}),
the first two (remaining) equations of system (\ref{second-order-system}) are given by
\begin{eqnarray}
\label{system-1-2}
\left\{ \begin{array}{l}  M_1 \frac{d^2 \gamma_1}{d t^2} = 6 P_2 \gamma_2^2, \\
M_2 \frac{d^2 \gamma_2}{d t^2} = 6 R_2 \gamma_2^2 + 12 P_2 \gamma_1 \gamma_2,
\end{array} \right.
\end{eqnarray}
The system is invariant on the subspace $S$ in (\ref{reduction}) if the constant $C$ is a solution
of the following quadratic equation:
$$
2 M_1 P_2  C^2 + M_1 R_2 C - M_2 P_2 = 0.
$$
The quadratic equation admits two nonzero real solutions $C$
if the discriminant is positive:
$$
\mathcal{D} := M_1^2 R_2^2 + 8 M_1 M_2 P_2^2 > 0,
$$
which is true thanks to the positivity of $M_1$ and $M_2$ in (\ref{positive-elements}).
The reduced system (\ref{second-order-system}) on the invariant subspace (\ref{reduction})
yields the following scalar second-order equation:
\begin{equation}
\label{gamma_1_evolution}
C^2 M_1 \frac{d^2 \gamma_1}{d t^2} - 6 P_2 \gamma_1^2 = 0,
\end{equation}
where $C \neq 0$, $M_1 > 0$ and $P_2 < 0$. The zero equilibrium $(\gamma_1,\dot{\gamma}_1) = (0,0)$
is a cusp point so that it is unstable in the nonlinear equation (\ref{gamma_1_evolution}).

Next, we prove the assertion of the lemma. For every sufficiently small $\delta>0$, we choose the initial point
$(\gamma(0), \beta(0)) \in \mathbb{R}^{N-1} \times \mathbb{R}^{N-1}$ in the invariant subspace
$S$ in (\ref{reduction}) satisfying $\|\gamma(0)\| + \|\beta(0)\| \leq \delta$.
Since $(0,0)$ is a cusp point in the reduced equation (\ref{gamma_1_evolution})
there exists a $t_0 > 0$ such that $\| \gamma(t_0) \| = \epsilon$ and $\| \gamma(t) \| > \epsilon$ for $t > t_0$
for any fixed $\epsilon > 0$.

Let us consider a fixed sufficiently small value of $\epsilon > 0$.
We have $\gamma(t) = \mathcal{O}(\epsilon)$ for $t \in [0,t_0]$ by the construction
Setting $\dot \gamma_1 = \beta_1$, we assume that $\beta_1(t) = \mathcal{O}(\epsilon^{3/2})$ for $t \in [0,t_0]$
by the choice of initial condition. The evolution equation (\ref{gamma_1_evolution}) implies that
for every $t\in [0, t_0]$ there is an $(\epsilon,\delta)$-independent constant $A>0$ such that
\begin{eqnarray*}
\left\{ \begin{array}{l}
|\gamma_1(t)| \leq \left| \int_0^t \beta_1(s) ds \right| + \left|\gamma_1(0) \right| \leq A \epsilon^{3/2} t_0 + \delta \\
|\beta_1(t)| \leq A \left| \int_0^t \gamma_1^2(s) ds \right| +
\left|\beta_1(0) \right| \leq A \epsilon^{2} t_0 + \delta.
\end{array} \right.
\end{eqnarray*}
If $\delta \in (0,A\epsilon^{3/2})$, then $\gamma_1(t) = \mathcal{O}(\epsilon)$ and $\beta_1(t) = \mathcal{O}(\epsilon^{3/2})$
remains for $t \in [0,t_0]$ with $t_0 = \mathcal{O}(\epsilon^{-1/2})$. The assertion of the lemma is proven.
\end{proof}

By Lemma \ref{Ham_instability}, there exists a trajectory of the finite-dimensional system
(\ref{normal-form-time}) near the zero equilibrium which leaves the $\epsilon$-neighborhood
of the zero equilibrium over the time span $[0,t_0]$ with
$t_0 = \mathcal{O}(\epsilon^{-1/2})$. Moreover, we have $\gamma(t) = \mathcal{O}(\epsilon)$ and
$\beta(t) = \mathcal{O}(\epsilon^{3/2})$ for every $t \in [0,t_0]$. This scaling suggests
to consider the following region in the phase space $\mathbb{R}^{N-1} \times \mathbb{R}^{N-1}$
of the evolution system (\ref{a_0_system-U-W}) in variables $(c, b)$:
\begin{equation}
\label{bound-coefficients-apriori}
\| c(t) \| \leq A \epsilon, \quad \| b(t) \| \leq A \epsilon^{3/2}, \quad t \in [0,t_0],
\end{equation}
where $t_0 \leq A \epsilon^{-1/2}$,
for an $\epsilon$-independent constant $A > 0$.
Vectors $(c,b)$ in the region (\ref{bound-coefficients-apriori})
still satisfy the bound (\ref{remainder_terms_bound}), hence
the decompositions (\ref{a_0_orth_decomposition}) and (\ref{a_0_r1})
remain valid due to the bound (\ref{bound-omega-U-W}) in Lemma \ref{lem-remainder-last}.
The following result proved in \cite{KP1} shows that a trajectory
of system (\ref{a_0_system-U-W}) follows closely to the trajectory of the finite-dimensional system (\ref{normal-form-time})
in the region (\ref{bound-coefficients-apriori}).

\begin{lemma}
\label{lem-correction}
For $\epsilon > 0$ sufficiently small,
assume that the remainder terms of the solution $\Psi$ decomposed as (\ref{a_0_orth_decomposition}) and
(\ref{a_0_r1}) satisfy (\ref{remainder_terms_bound}) and let $(\gamma,\beta) \in C^1([0,t_0],\mathbb{R}^{N-1} \times \mathbb{R}^{N-1})$
be a solution to the reduced system (\ref{normal-form-time}) in the region (\ref{bound-coefficients-apriori}).
The solution $(c,b) \in C^1([0,t_0],\mathbb{R}^{N-1} \times \mathbb{R}^{N-1})$
to the evolution system (\ref{a_0_system-U-W}) with initial datum $c(0) = \gamma(0)$ and $b(0) = \beta(0)$
remains in the region (\ref{bound-coefficients-apriori}) and there exist a generic $\epsilon$-independent constant
$A > 0$ such that
\begin{equation}
\label{bound-coefficients}
\| c(t) - \gamma(t) \| \leq A \epsilon^{3/2}, \quad \| b(t) - \beta(t) \| \leq A \epsilon^2, \quad t \in [0,t_0].
\end{equation}
\end{lemma}

\begin{remark}
By the bound (\ref{bound-coefficients}), $c_1(t)$ closely follows $\gamma_1(t)$, whereas
by the first equation of system (\ref{system-1-2}) with $M_1 > 0$ and $P_2 < 0$,
the map $t \mapsto \gamma_1(t)$ is monotonically decreasing if $\gamma_1'(0) \leq 0$.
Thanks to the correspondence between $c_1(t)$ and $a(t)$ in Remark \ref{remark-a},
this corresponds to the irreversible drift of the parameter $a(t)$ towards smaller values
in Corollary \ref{corollary-a}. Compared to Corollary \ref{corollary-a}, the irreversible drift
is observed from the half-soliton state with $a(0) = a_0 = 0$.
\end{remark}

\subsection{Proof of Theorem \ref{instability_result}}

Let us consider the unstable solution of the reduced system (\ref{normal-form-time})
according to Lemma \ref{Ham_instability}. This solution belongs to the region (\ref{bound-coefficients-apriori})
and, by Lemma \ref{lem-correction}, the correction terms satisfy (\ref{bound-coefficients}).
Therefore, the solution of system (\ref{a_0_system-U-W}) still satisfies the bound (\ref{bound-coefficients-apriori})
over the time span $[0,t_0]$ with $t_0 = \mathcal{O}(\epsilon^{-1/2})$.

For every $\epsilon > 0$, we set $\delta \in (0,A\epsilon^{3/2})$ and use Lemma \ref{lem-remainder-last}
to control the terms $\omega$, $U^\perp$ and $W^\perp$ by the bound
(\ref{bound-omega-U-W}). Therefore, the solution given by decompositions (\ref{a_0_orth_decomposition}) and (\ref{a_0_r1})
satisfies the bound (\ref{delta_0_bound}) for $t \in [0,t_0]$.

Since the solution $\gamma$ to the reduced system (\ref{normal-form-time})
grows in time and reaches the boundary in the region (\ref{bound-coefficients-apriori})
by Lemma \ref{Ham_instability}, the bound (\ref{bound-coefficients}) implies that
the same is true for the solution $c(t)$ of system (\ref{a_0_system-U-W}).
Hence, for every fixed $\epsilon > 0$ (sufficiently small), the initial data $\Psi_0$
satisfying the bound (\ref{a_0_initial_datum}) with $\delta \in (0,A\epsilon^{3/2})$ generates
the unique solution $\Psi$ of the NLS equation (\ref{eq1}) which
reaches and crosses the boundary in (\ref{orbital-instab}) for some $t_0 = \mathcal{O}(\epsilon^{-1/2})$.
Theorem \ref{instability_result} is proved.

\section{Numerical verification}
\label{sec-numerics}

Here we describe numerical experiments that illustrate the implications of Theorems \ref{main_theorem_1} and \ref{instability_result}.
Note that in all figures, we plot $U = (u_1,u_2,\ldots,u_N)$ with the components
$$
u_j(x) = \alpha_j \psi_j(x).
$$
Under the boundary conditions in (\ref{H1}), the solution $U$ is continuous across the vertex.
The shifted state of Lemma \ref{solutions-II} with parameter $a \in \mathbb{R}$ is given in the variable $U$
by $\alpha \Phi(\cdot;a)$.

\subsection{Numerical method}

We briefly describe the numerical method used to simulate time-dependent solutions
of the NLS equation (\ref{eq1}).
Each semi-infinite edge is truncated to a finite length $L$, with Dirichlet boundary
conditions imposed at the leaf endpoint. The spatial derivatives are discretized using
second-order centered differences. The derivative boundary conditions in~\eqref{H2}
are enforced using ghost points. That is, if the spacing between the grid points
on edge $j$ is given by $\Delta x_j$, then the grid points are located at
$x_k^{(j)}$ = $(k-\frac{1}{2})\Delta x_j$. There is no grid point at the vertex:
instead along each edge there is a \emph{ghost point} located at $x_0^{(1)} = \Delta x_1/2$
and $x_0^{(j)} = -\Delta x_j/2$ for $j = 2,...,N$.
The values at the ghost points are determined by enforcing the discretized  boundary condition
with the value at the vertex approximated by linear interpolation.

Time evolution is calculated using a second-order split-step method following Weidemann and Herbst~\cite{WeiHer:1986}.
The linear and nonlinear parts of the evolution are handled separately, with an explicit phase rotation
with time step $\Delta t$ sandwiched between two steps of the linear part with time steps of $\Delta t/2$
using the Crank-Nicholson scheme.

While the simulations are primarily concerned with the behavior of solutions concentrated away from the computational boundary,
the evolution naturally gives rise to radiation, which quickly propagates to the boundary. This radiation
interacts with the computational boundaries and leads to numerical instability that effects the computational results.
Following Nissen and Kreiss~\cite{NisKre2015}, perfectly matched layers are used at the leaf
endpoints to absorb this radiation. In practice, this leads to a modification of the discretized
Laplacian at a finite number of points near the leaf vertices.

The standard care is taken to ensure the accuracy of the numerical simulations including quantitative convergence study, and the tracking of conserved quantities.

\subsection{Experiments with eigenfunction perturbation}

We consider the balanced star graph with one incoming edge and two outgoing edges.
For simplicity, we also assume $\alpha_2=\alpha_3$.
For the shifted state of Lemma \ref{solutions-II},
the spectrum of linearized operator $L_+(\omega,a)$ is
shown on Fig. \ref{fig-2}. In particular, the operator possesses
a negative eigenvalue $\lambda_0=-3$ for all $a$, and a second eigenvalue $\lambda_1(a)$ for $a<a_*$,
which is positive for $0<a<a_*$ and negative for $a < 0$.
Thus the Morse index suggests that the shifted state
is linearly stable when $a>0$ and linearly unstable when $a<0$.

Let $U_a$ be the eigenfunction of the operator $L_+(\omega,a)$ associated with
the eigenvalue $\lambda_1$, which exists for every $a < a_*$. Let it be normalized
by $\norm{U_a}{2} = 1$. We consider the initial datum to the NLS equation (\ref{eq1}) of the form
\begin{equation}
\Psi_0 = \Phi(\cdot;a) + \epsilon U_a.
\label{initcond}
\end{equation}
We assume here that $U_a(x) = 0$ on edge one, $U_a(x) < 0$ on edge two, and $U_a(x) > 0$ on edge three.
Such an initial datum has the initial momentum $P(\Psi_0) = 0$
regardless of whether $\epsilon$ is real or has nonzero imaginary part.

We first present a simulation of the unstable shifted state with $a = -0.55$ and $\epsilon=0.1$, in which the non-monotonic part
lies on the two outgoing edges. The time-dependent solution
is plotted in Figure~\ref{fig:eig045pcolor}. The solution on edge three is initially slightly larger than on edge two.
This asymmetry grows until the solution has concentrated on edge three, and then propagates away from the vertex along edge three.
A lower-amplitude traveling wave, not visible on this plot, propagates away from the vertex on edge two.

\begin{figure}[htbp] 
   \centering
   \includegraphics[width=3in]{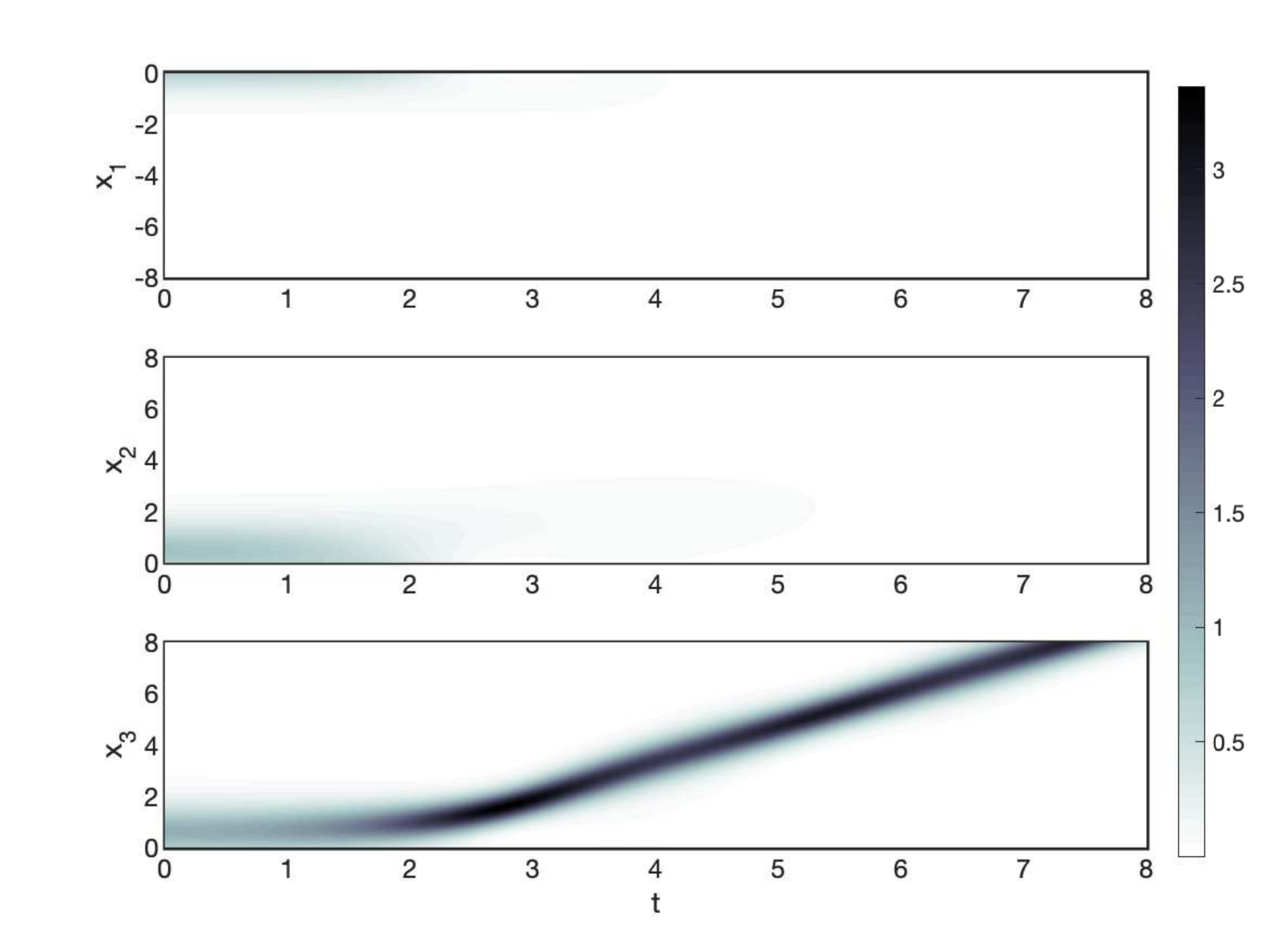}
   \caption{A numerical solution with initial datum~\eqref{initcond} for $a=-0.55$ and $\epsilon=0.1$.
   The colorbar corresponds to values of $\abs{u}^2$. The three panels correspond to the solution on edges 1, 2, and 3 going down.}
   \label{fig:eig045pcolor}
\end{figure}

\begin{figure}[htbp] 
   \centering
   \includegraphics[height=0.35\textwidth]{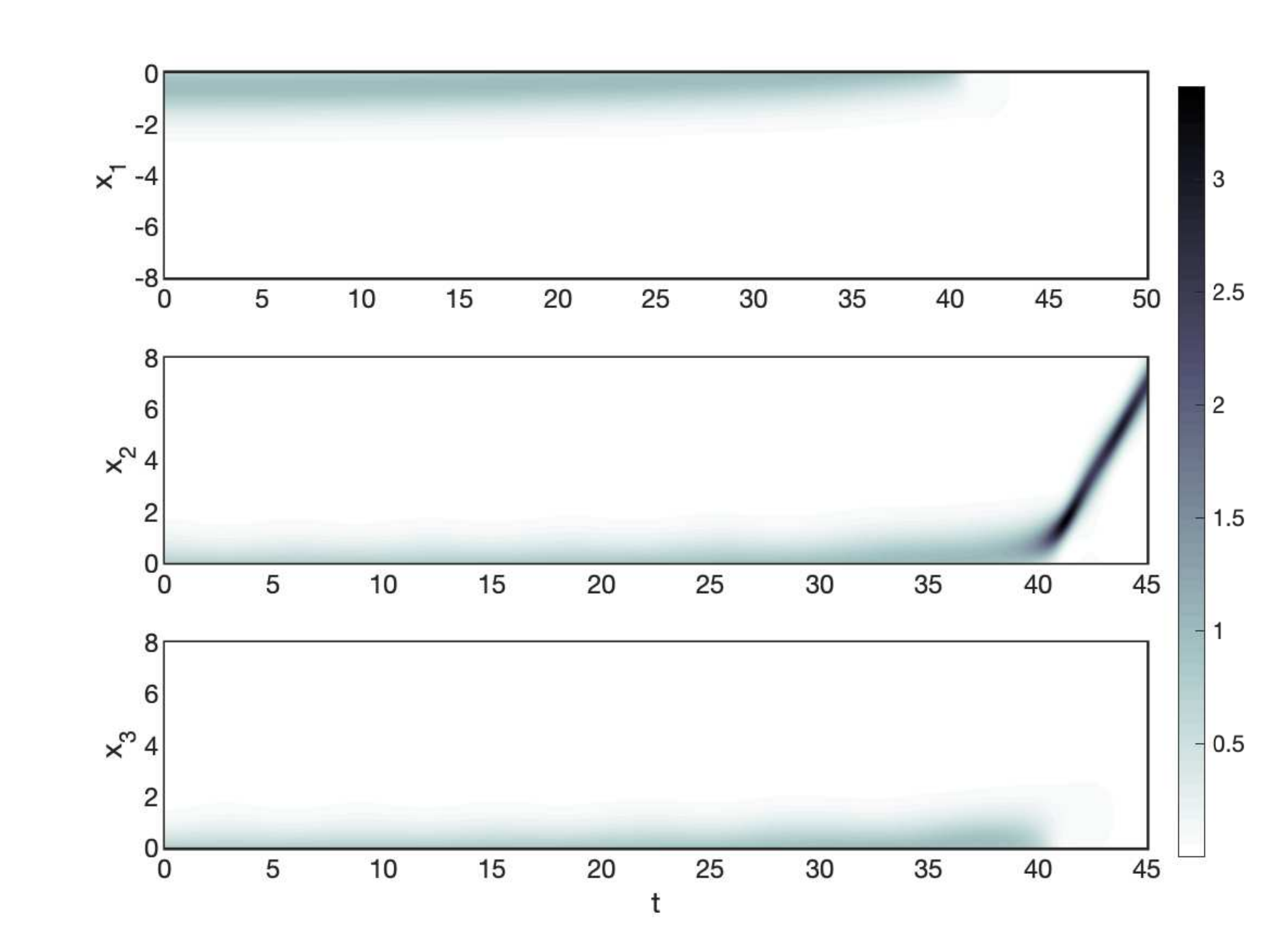}
   \includegraphics[height=0.35\textwidth]{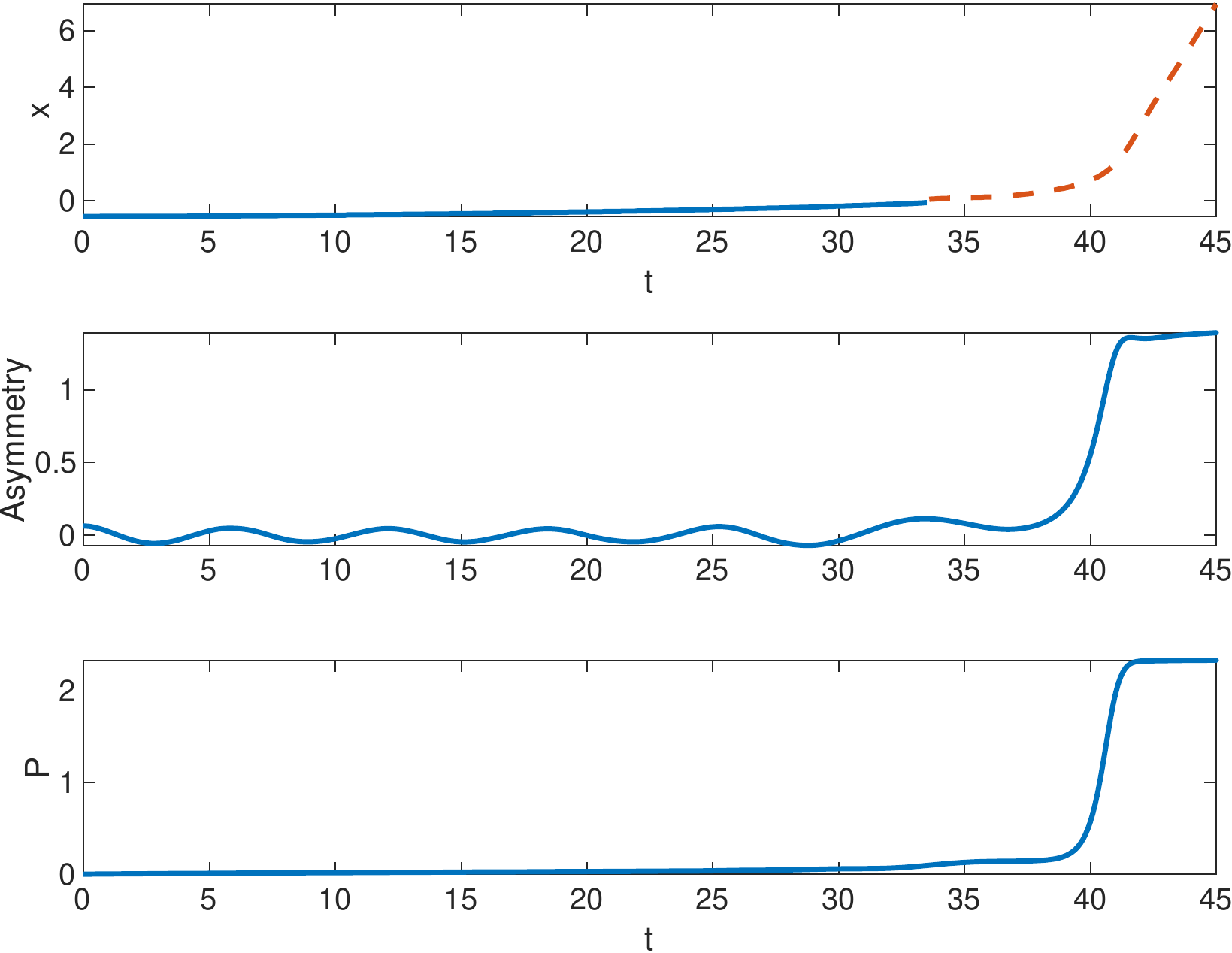}
   \caption{Left: A numerical solution with initial datum~\eqref{initcond} for $a=0.55$ and $\epsilon=0.1$.
   The colorbar corresponds to values of $\abs{u}^2$. Right: Postprocessed quantities form the same simulation.
   (Top) The position of the maximum of $u$. The solid line for $t<33.5$ describes the position on the incoming edge one.
   The dashed line for $t > 33.5$ shows the position of the maximum on edge two. (Middle) The asymmetry, defined as
   $\| u_2 \|_{L^2(\mathbb{R}^+)}-\| u_3 \|_{L^2(\mathbb{R}^+)}$.
   (Bottom) The momentum $P(\Psi)$ versus time $t$.}
   \label{fig:eig046pcolor}
\end{figure}

The behavior is more interesting when we consider the stable shifted state with $a =0.55>0$ and $\epsilon=0.1$, in which the non-monotonic part
lies on the only incoming edge. The shifted state is linearly stable but Theorem \ref{main_theorem_1}
predicts drift of this shifted state towards smaller
values of $a$, where Theorem \ref{instability_result} predicts nonlinear instability once the vertex is reached.
The left panels of Fig.~\ref{fig:eig046pcolor} show that for $0<t \lesssim 40$, the dynamics is very slow. After $t \approx 40$,
the shifted state quickly propagates away from the vertex along edge two.

The right panels of Fig.~\ref{fig:eig046pcolor} show some quantities post-processed from the simulation.
The top panel shows the location of the numerical maximum of $|u|$. At about $t=33.5$, the maximum crosses
from edge one to edge two at $x=0$. The second panel shows the $\| u_2 \|_{L^2(\mathbb{R}^+)} - \| u_3 \|_{L^2(\mathbb{R}^+)}$,
the difference between $L^2$ norms along the two outgoing edges, which is used as a proxy for
the amplitude of the eigenfunction perturbation. This quantity oscillates between positive and negative values
while the shifted state lies along the incoming edge, corresponding to stable evolution. It enters a period
of exponential growth for $38\lesssim t\lesssim 40$, once the shifted state has moved to the outgoing edges.
During this period, a high-amplitude state forms on edge two and a low-amplitude state on edge three.
The bottom panel shows that the momentum grows slowly until $t \approx 33.5$, then grows quickly,
before saturating at $t\approx 42$, and moving at constant momentum thereafter.

This numerical simulation shows that perturbations to a shifted state with $a>0$ grow sub-exponentially,
consistent with the drift instability of Theorem \ref{main_theorem_1}. Similarly, postprocessing the simulation
shown in Fig.~\ref{fig:eig045pcolor} with $a<0$ shows that the perturbation immediately grows at an exponential rate,
consistent with a linear instability proven in \cite{KP2}.

\subsection{Experiments with other perturbations}
\label{sec:anotherFamily}

Initial data of type~\eqref{initcond} only exist for $a < a_* \approx 0.66$ if $\omega = 1$. However we found almost
identical dynamics as above by perturbing a shifted state with $a > a_*$, by a function similar to $U_a$
which vanishes on the incoming edge, and takes the form $\pm c x e^{-\lambda x}$ on the two outgoing edges.
We will not report further on such simulations.

All initial data of type~\eqref{initcond} have zero initial momentum $P(\Psi_0) = 0$.
An interesting question is what happens when we apply a perturbation such that the initial datum
gives $P(\Psi_0) < 0$. Equation~\eqref{dmom_dt} implies that the map $t \mapsto P(\Psi)$ is increasing.
Therefore, the question is if the shifted states propagating along the incoming edge away from the vertex initially
can escape the drift instability. Note that the condition of Theorem \ref{main_theorem_1} requires $P(\Psi_0) > 0$.

We construct perturbed shifted states as follows. Choose  $\mu \in \mathbb{C}$ and define $\Psi_0$ by
\begin{equation}
\left\{ \begin{array}{l}
\psi_1(x) = \alpha_1^{-1} e^{\mu x}\phi(x+a) \\
\psi_2(x) = \alpha_2^{-1} e^{2\mu x} \phi(x+a), \\
\psi_3(x) = \alpha_3^{-1} \phantom{e^{2\mu x}} \phi(x+a).
\end{array} \right.
\label{nonconserving}
\end{equation}
If $\mu$ is small, then $\Psi_0 \in H_\Gamma^2$ and $\lVert \Psi_0 - \Phi(\cdot;a) \rVert$ is small.
The initial datum has nonzero initial momentum $P(\Psi_0) \neq 0$ if ${\rm Im}(\mu) \neq 0$.

We perform such a simulation with $a=-1$ and pure imaginary $\mu=0.1 i$. This slowly modulates
the phase of the shifted state while keeping the amplitude at each point unchanged.
The solution has its maxima on the outgoing edges and has negative initial momentum $P(\Psi_0) < 0$,
so initially it propagates toward the vertex. Fig.~\ref{fig:modified050pcolor} shows
that the numerical solution quickly concentrates on edge three and begins propagating away
from the vertex. This is more visible in the right panel, which shows the location of the
maximum value on edge three, which initially decreases, before increasing, and the momentum,
which is initially negative, but which rapidly becomes positive.

\begin{figure}[htbp] 
   \centering
   \includegraphics[height=0.35\textwidth]{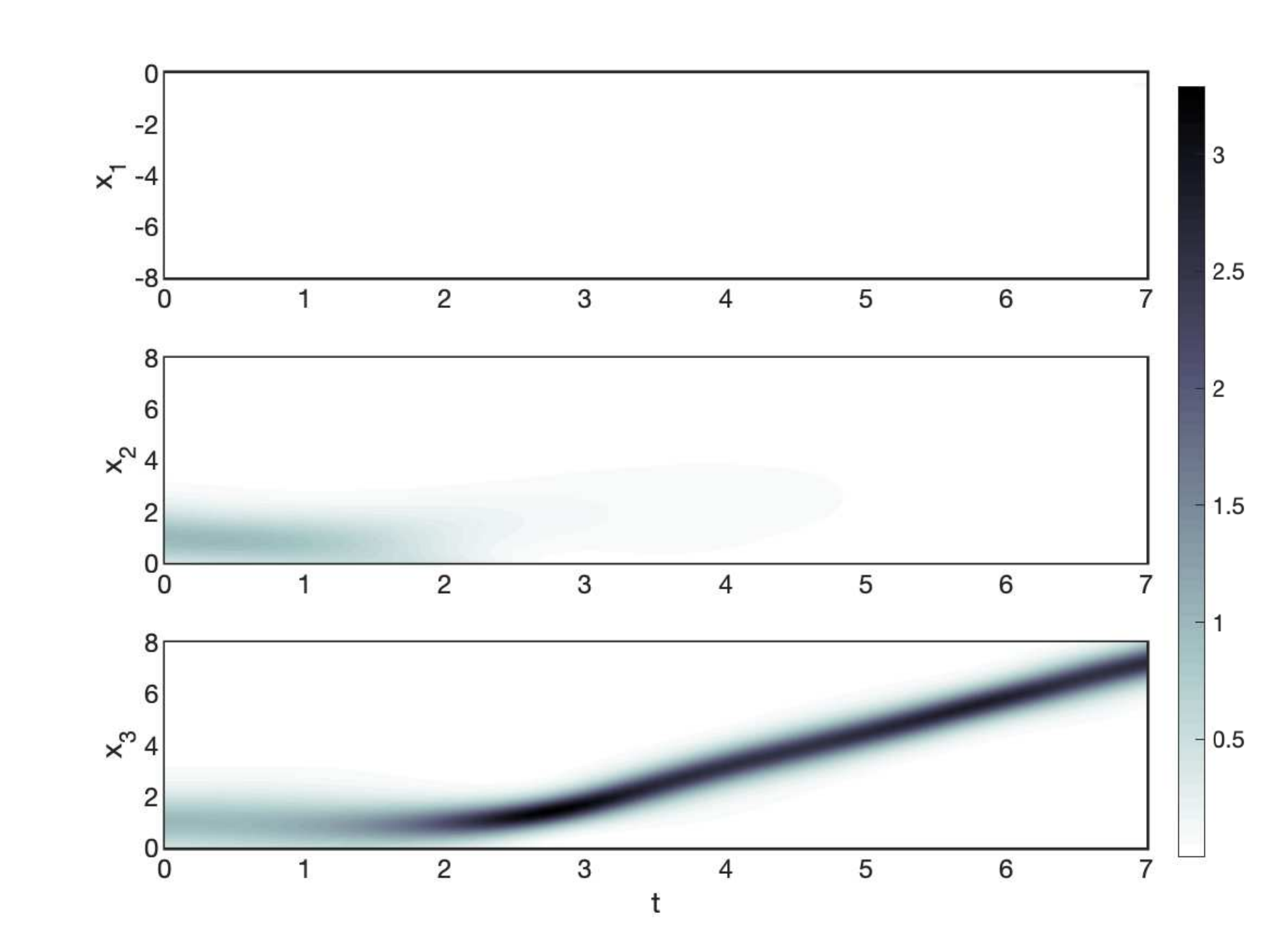}
   \includegraphics[height=0.35\textwidth]{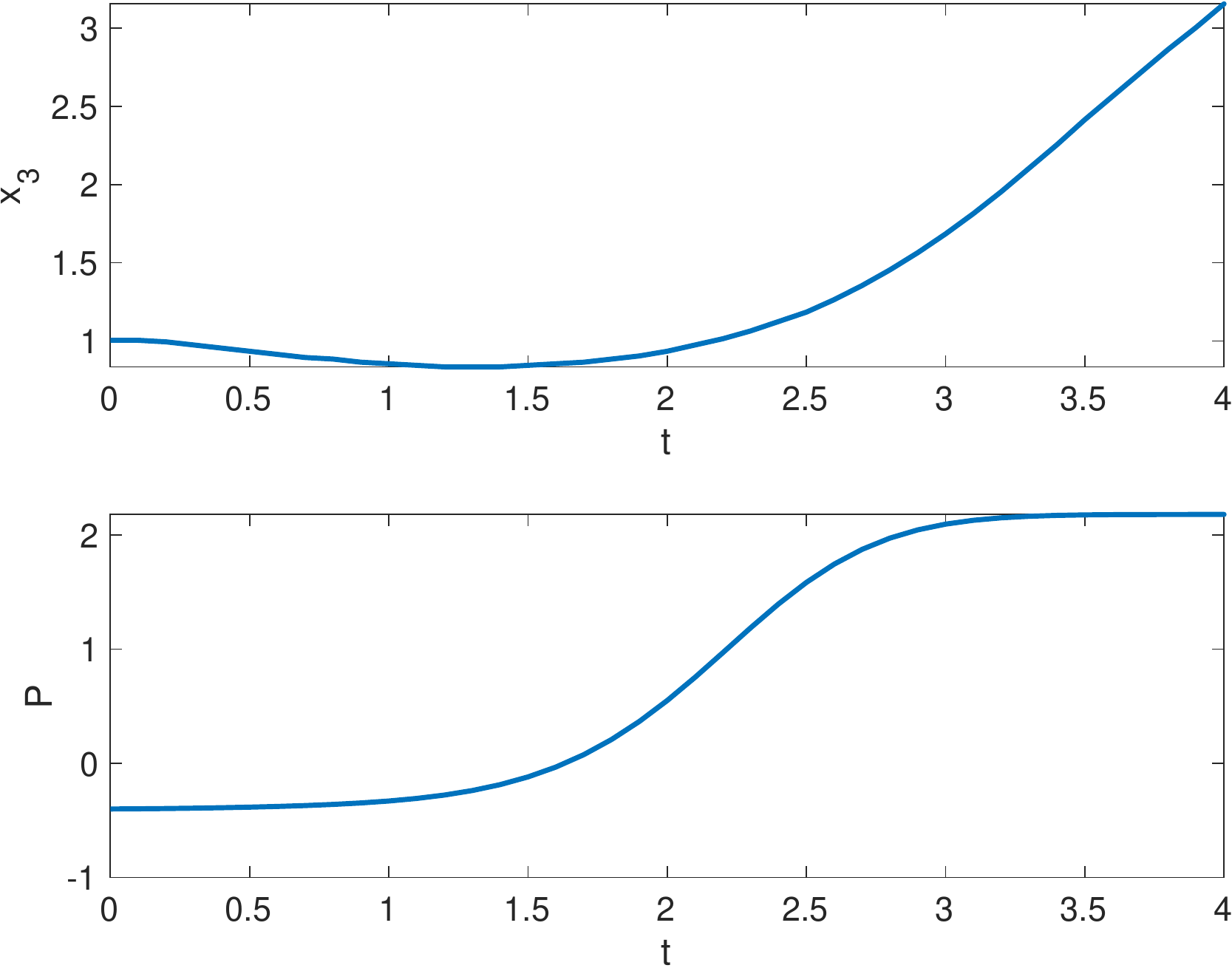}
   \caption{A numerical solution with initial datum \eqref{nonconserving}
   for $a = -1$ and $\mu = 0.1 i$.  Details as in Figure~\ref{fig:eig046pcolor}.}
   \label{fig:modified050pcolor}
\end{figure}

\begin{figure}[htbp] 
   \centering
   \includegraphics[width=3in]{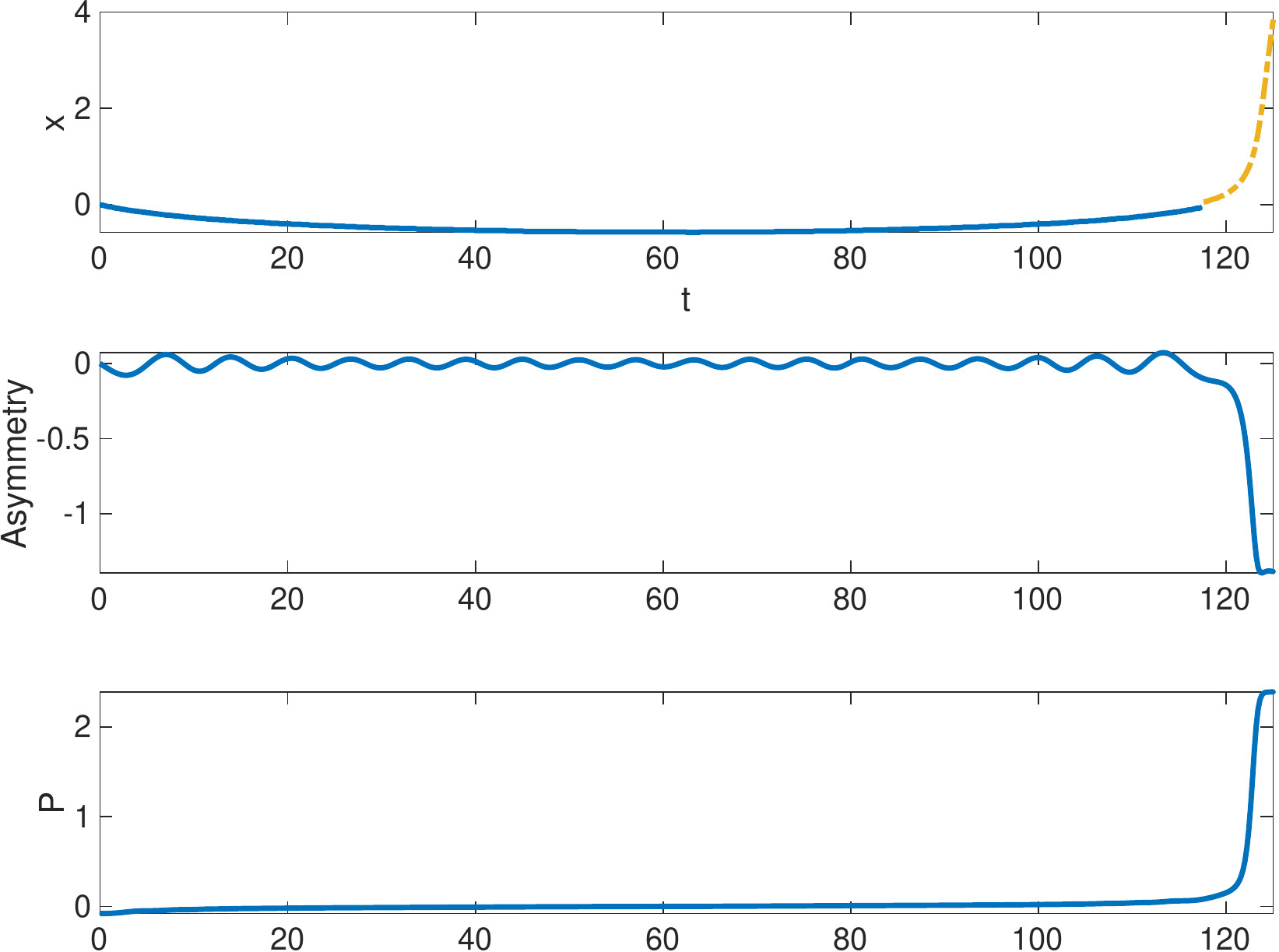}
   \caption{A numerical solution with initial datum \eqref{nonconserving}
   for $a = 0$ and $\mu = -0.02 i$. (Top) The position of the maximum of $\abs{u}$, on edge one
   for $t<117$ and on edge three (dashed) for $t>117$. (Middle) Asymmetry of the solution between the two outgoing edges.
   (Bottom) The momentum $P(\Psi)$ versus time $t$.}
   \label{fig:modified065post}
\end{figure}

A final simulation shows a solitary wave that travels away from the vertex along the incoming edge
before reversing direction, crossing the vertex and escaping to infinity along one of the outgoing edges;
shown in Fig.~\ref{fig:modified065post}. The simulation was performed with initial datum
of form~\eqref{nonconserving} with $a=0$ and $\mu = -0.02\ i$. The initial datum differs from
the half-soliton state $\Phi$ by $0.04$ in $L^2$-norm, hence it can be considered as a small perturbation of
the half-soliton state. The initial momentum is $P(\Psi_0)=-0.08$. The solution gradually slows down,
with the momentum vanishing at about $t=62$. The maximum crosses the vertex at about $t=117$ and
at this point the solution concentrates on edge 3 and the momentum begins increasing rapidly.

\section{Conclusion}
\label{sec-conclusion}

This paper concludes the study of the NLS equation on a balanced star graph originated in \cite{KP1,KP2}.
We have proven analytically and illustrated numerically the conjectures formulated in \cite{KP2},
namely that the perturbation that breaks a translational symmetry of the shifted state
induces instability of the shifted state in the time evolution. When the shifted state has monotonic tail
on the only incoming edge, it is spectrally unstable and the perturbations grow exponentially fast.
When the shifted state has monotonic tails on the $(N-1)$ outgoing edges, it is spectrally stable
but nonlinearly unstable. The perturbations do not grow in time but the center of mass of the shifted
state drifts slowly along the incoming edge towards the vertex. This drift is induced
by the irreversible growth of momentum in the time evolution due to the broken translational symmetry
at the vertex point. Once the center of mass for the shifted state reaches the vertex, the
perturbations start to grow faster, first algebraically and then exponentially.
The numerical simulations not only verify the outcomes of the nonlinear dynamics predicted by the main theorems
(Theorems \ref{main_theorem_1} and \ref{instability_result}) but also illustrate how generic
this phenomenon is on the balanced star graphs.

\end{document}